\documentclass[11pt,english]{amsart}
\usepackage[T1]{fontenc}
\usepackage[latin1]{inputenc}
\usepackage{amstext}
\usepackage{amsthm}
\usepackage{amssymb}

\makeatletter

\providecommand{\tabularnewline}{\\}

\numberwithin{equation}{section}
\numberwithin{figure}{section}
\theoremstyle{plain}
\newtheorem{thm}{\protect\theoremname}
\theoremstyle{remark}
\newtheorem{rem}[thm]{\protect\remarkname}
\theoremstyle{plain}
\newtheorem{cor}[thm]{\protect\corollaryname}
\theoremstyle{plain}
\newtheorem{prop}[thm]{\protect\propositionname}
\theoremstyle{plain}
\newtheorem{lem}[thm]{\protect\lemmaname}
\theoremstyle{definition}
\newtheorem{example}[thm]{\protect\examplename}

\usepackage{tikz}
\usepackage{pgfplots}

\usepackage{babel}
\providecommand{\corollaryname}{Corollary}
\providecommand{\examplename}{Example}
\providecommand{\lemmaname}{Lemma}
\providecommand{\propositionname}{Proposition}
\providecommand{\remarkname}{Remark}
\providecommand{\theoremname}{Theorem}

\usepackage{babel}
\providecommand{\corollaryname}{Corollary}
\providecommand{\examplename}{Example}
\providecommand{\lemmaname}{Lemma}
\providecommand{\propositionname}{Proposition}
\providecommand{\remarkname}{Remark}
\providecommand{\theoremname}{Theorem}

\makeatother

\usepackage{babel}
\providecommand{\corollaryname}{Corollary}
\providecommand{\examplename}{Example}
\providecommand{\lemmaname}{Lemma}
\providecommand{\propositionname}{Proposition}
\providecommand{\remarkname}{Remark}
\providecommand{\theoremname}{Theorem}

\begin{document}
\addtolength{\textwidth}{0mm} \addtolength{\hoffset}{-0mm} \addtolength{\textheight}{0mm}
\addtolength{\voffset}{-0mm}

%\subjclass{Primary: 14J29} ; Secondary: 14G10, 14G15}

\global\long\def\CC{\mathbb{C}}%
\global\long\def\BB{\mathbb{B}}%
\global\long\def\PP{\mathbb{P}}%
\global\long\def\QQ{\mathbb{Q}}%
\global\long\def\RR{\mathbb{R}}%
\global\long\def\FF{\mathbb{F}}%
\global\long\def\DD{\mathbb{D}}%
\global\long\def\NN{\mathbb{N}}%
\global\long\def\ZZ{\mathbb{Z}}%
\global\long\def\HH{\mathbb{H}}%
\global\long\def\Gal{{\rm Gal}}%
\global\long\def\OO{\mathcal{O}}%
\global\long\def\pP{\mathfrak{p}}%
\global\long\def\bA{\mathbf{A}}%
\global\long\def\pPP{\mathfrak{P}}%
\global\long\def\qQ{\mathfrak{q}}%
\global\long\def\cC{\mathfrak{\mathcal{C}}}%
\global\long\def\mm{\mathcal{M}}%
\global\long\def\calK{\mathcal{K}}%
\global\long\def\aaa{\mathfrak{a}}%
\global\long\def\a{\alpha}%
\global\long\def\b{\beta}%
\global\long\def\d{\delta}%
\global\long\def\D{\Delta}%
\global\long\def\L{\Lambda}%
\global\long\def\g{\gamma}%
\global\long\def\G{\Gamma}%
\global\long\def\d{\delta}%
\global\long\def\D{\Delta}%
\global\long\def\e{\varepsilon}%
\global\long\def\k{\kappa}%
\global\long\def\l{\lambda}%
\global\long\def\m{\mu}%
\global\long\def\o{\omega}%
\global\long\def\p{\pi}%
\global\long\def\P{\Pi}%
\global\long\def\s{\sigma}%
\global\long\def\S{\Sigma}%
\global\long\def\t{\theta}%
\global\long\def\T{\Theta}%
\global\long\def\f{\varphi}%
\global\long\def\deg{{\rm deg}}%
\global\long\def\det{{\rm det}}%
\global\long\def\Dem{Proof: }%
\global\long\def\ker{{\rm Ker\,}}%
\global\long\def\im{{\rm Im\,}}%
\global\long\def\rk{{\rm rk\,}}%
\global\long\def\car{{\rm car}}%
\global\long\def\fix{{\rm Fix( }}%
\global\long\def\card{{\rm Card\  }}%
\global\long\def\codim{{\rm codim\,}}%
\global\long\def\coker{{\rm Coker\,}}%
\global\long\def\mod{{\rm mod }}%
\global\long\def\pgcd{{\rm pgcd}}%
\global\long\def\ppcm{{\rm ppcm}}%
\global\long\def\la{\langle}%
\global\long\def\ra{\rangle}%
\global\long\def\Alb{{\rm Alb(}}%
\global\long\def\Jac{{\rm Jac(}}%
\global\long\def\Disc{{\rm Disc(}}%
\global\long\def\Tr{{\rm Tr(}}%
\global\long\def\NS{{\rm NS(}}%
\global\long\def\Pic{{\rm Pic(}}%
\global\long\def\Pr{{\rm Pr}}%
\global\long\def\Km{{\rm Km}}%
\global\long\def\rk{{\rm rk(}}%
\global\long\def\Hom{{\rm Hom(}}%
\global\long\def\End{{\rm End}}%
\global\long\def\aut{{\rm Aut}}%
\global\long\def\SSm{{\rm S}}%
\global\long\def\psl{{\rm PSL}}%
\global\long\def\cu{{\rm (-2)}}%
\global\long\def\aut{{\rm Aut}}%
\global\long\def\mod{{\rm \,mod\,}}%
\title[Kummer structures on generalized Kummer surfaces]{Constructions of Kummer structures on generalized Kummer surfaces}
\author{Xavier Roulleau, Alessandra Sarti}
\begin{abstract}
We study generalized Kummer surfaces $\Km_{3}(A)$, by which we mean
the K3 surfaces obtained by desingularization of the quotient of an
abelian surface $A$ by an order $3$ symplectic automorphism group.
Such a surface carries $9$ disjoint configurations of two smooth
rational curves $C,C'$ with $CC'=1$. This $9{\bf A}_{2}$-configuration
plays a role similar to the Nikulin configuration of $16$ disjoint
smooth rational curves on (classical) Kummer surfaces. We study the
(generalized) question of T. Shioda: suppose that $\Km_{3}(A)$ is
isomorphic to $\Km_{3}(B)$, does that imply that $A$ and $B$ are
isomorphic? We answer by the negative in general, by two methods:
by a link between that problem and Fourier--Mukai partners of $A$,
and by construction of $9{\bf A}_{2}$-configurations on $\Km_{3}(A)$
which cannot be exchanged under the automorphism group. 
\end{abstract}

\maketitle

\section{Introduction}

A Kummer surface $\Km(A)$ is the minimal desingularization of the
quotient of an abelian surface $A$ by the standard involution $[-1]$.
It is a K3 surface containing $16$ disjoint $\cu$-curves, which
lie over the $16$ singularities of $A/\langle[-1]\rangle$. Such
set of curves is called a Kummer (or $16{\bf A}_{1}$) configuration.
A well-known result of Nikulin \cite{NikulinK} gives the converse:
if a K3 surface contains a $16{\bf A}_{1}$-configuration, then it
is the Kummer surface of an abelian surface $A$, such that the $16$
$\cu$-curves lie over the singularities of $A/\langle[-1]\rangle$.

In $1977$ Shioda \cite{Shioda} asked the following question: \textit{if
two abelian surfaces $A$ and $B$ satisfy $\Km(A)\simeq\Km(B)$,
is it true that $A\simeq B$ ?}

Gritsenko and Hulek \cite{GriHu} gave a negative answer to that question
in general. In \cite{RS1,RS2}, we studied and constructed examples
of two $16{\bf A}_{1}$-configurations on the same Kummer surface
such that their associated abelian surfaces are not isomorphic.

Kummer surfaces have natural generalizations to quotients of an abelian
surface $A$ by other symplectic groups $G\subseteq\aut(A)$. If $G\cong\ZZ/3\ZZ$,
then the quotient surface $A/G$ for the action of $G$ on $A$ has
$9$ cusp singularities, in bijection with the fixed points of $G$.
Its minimal desingularization, denoted by $\Km_{3}(A)$, is a K3 surface
which contains what we call a \textit{generalized Kummer configuration}
(or $9{\bf A}_{2}$\textit{-configuration}), which means that the
surface contains $9$ disjoint $\mathbf{A}_{2}$-configurations, i.e.\/
pairs $(C,C')$ of $\cu$-curves such that $CC'=1$. Barth \cite{Barth}
proved that if a K3 surface contains a $9{\bf A}_{2}$-configuration,
then there exists an abelian surface $A$ and a symplectic order $3$
automorphism group such that $X=\Km_{3}(A)$. It is then natural to
ask the \textit{generalized Shioda's question}: does an isomorphism
$\Km_{3}(A)\simeq\Km_{3}(B)$ between two generalized Kummer surfaces
implies that the abelian surfaces $A$ and $B$ are isomorphic?

A generalized Kummer structure on a K3 surface $X$ is an isomorphism
class of pairs $(A,G)$ of abelian surfaces equipped with an order
$3$ symplectic automorphism subgroup $G\subset\aut(A)$, such that
$X\simeq\Km_{3}(A)$, where $\Km_{3}(A)$ is the minimal desingularization
of $A/G$. Thus Shioda's question is if there is only one generalized
Kummer structure on $X$. In \cite{Roulleau2}, we study the number
of such Kummer structures. In \cite{KRS}, we proved that there is
a one-to-one correspondence between Kummer structures on $X$ and
$\aut(X)$-orbits of $9{\bf A}_{2}$-configurations. In the present
paper, we obtain the first explicit examples of generalized Kummer
surfaces which possess two distinct generalized Kummer structures.
For that aim, we construct two $9\mathbf{A}_{2}$-configurations $\mathcal{C},\,\mathcal{C}'$
on the Kummer surface, and prove that there is no automorphism sending
one configuration to the other. A generalized Kummer surface $X=\Km_{3}(A)$
has a natural $9\mathbf{A}_{2}$-configuration 
\[
\mathcal{C}=\{A_{1},B_{1},\dots,A_{9},B_{9}\}.
\]
We suppose that $X$ is generic projective, so that its Picard number
is $19$. That hypothesis is assumed in all the paper. Let $L$ be
the big and nef generator of the orthogonal complement of the curves
in the Néron-Severi group. By a result of Barth \cite{Barth}, one
has either $L^{2}=6k+2$ or $L^{2}=6k$, for $k$ an integer. We suppose
that $6L^{2}$ is not a square, so that the two Pell-Fermat equations
\[
x^{2}-12(3k+1)y^{2}=1\,\text{ and }\,\,x^{2}-4ky^{2}=1
\]
have non-trivial solutions. Let us denote by $(x_{0},y_{0})$ the
fundamental solution according to these cases and let us define accordingly:
\[
\begin{array}{l}
B_{1}'=3y_{0}L-(\tfrac{1}{2}(x_{0}+1)A_{1}+x_{0}B_{1})\text{ if }L^{2}=6k+2,\\
B_{1}'=y_{0}L-(\tfrac{1}{2}(x_{0}+1)A_{1}+x_{0}B_{1})\,\,\text{ if }\,L^{2}=6k.
\end{array}
\]
The class $B_{1}'$ is a $\cu$-class in the Néron-Severi group of
$X$ (i.e. $B_{1}'^{2}=-2$) such that $B_{1}'A_{1}=B_{1}A_{1}=1$.
Our main result is 
\begin{thm}
\label{thm:Main}Suppose that $L^{2}=2t$ is such that either $L^{2}=2\mod6$
or $L^{2}\neq0\mod18$ or $3|y_{0}$. Then $B_{1}'$ is the class
of a $\cu$-curve such that $B_{1}'A_{1}=1$ and the $18$ $\cu$-curves
\[
\mathcal{C}'=\{A_{1},B_{1}',A_{2},B_{2},\dots,A_{9},B_{9}\}
\]
form a $9\mathbf{A}_{2}$-configuration. \\
 Suppose moreover that $L^{2}=2\mod6$ and $x_{0}\neq\pm1\mod2t$,
or $L^{2}=6\text{ or }12\mod$ $18$ and $x_{0}\neq\pm1\mod2k$. There
are no automorphisms sending $\mathcal{C}$ to $\mathcal{C}'$. For
these cases, there are (at least) two generalized Kummer structures
on the generalized Kummer surface $X$. 
\end{thm}

As pointed out in Example \ref{eq:ExamplesLess200} the first values
of $L^{2}$ for which our theorem produces new generalized Kummer
structures are: 
\[
20,44,68,84,92,104,110,116,120,126,132,140,164,168,176,188.
\]
Theorem \ref{smallL} shows that the hypothesis $x_{0}\neq\pm1\mod2k$
is sharp, since if $x_{0}=\pm1\mod2k$ and $L^{2}\leq200$, there
exists an automorphisms sending $\mathcal{C}$ to $\mathcal{C}'$.
We expect that one can remove the somehow technical hypothesis $L^{2}\neq0\mod18$
or $3|y_{0}$ of Theorem \ref{thm:Main}.

We observe that our construction can be applied to other configurations
$A_{j},B_{j}$ than $A_{1},B_{1}$. For $L^{2}\leq200$, one can check
that repeating twice the construction, we go back (up to automorphisms)
to the original Kummer configuration. 

The paper is structured as follows: in Section 2, we give a precise
description of the Néron-Severi group of a generalized Kummer surface
and how a divisor can be written in the $\QQ$-basis $L,A_{1},B_{1},\dots,A_{9},B_{9}.$
In the case that $L^{2}=2\mod6$, we also obtain that $\Km_{3}(A)\simeq\Km_{3}(B)$
if and only if $A$ and $B$ are Fourier--Mukai partners. In Section
3, according to the cases $L^{2}=0\text{ or }2\mod6$, we construct
two generalized Kummer configurations. In Section $4$, we give a
sufficient condition on which these Kummer configurations give rise
to two generalized Kummer structures. In the last section we describe
the projective model of the K3 surface determined by $L$ and we recall
some known constructions in the literature. We describe more in details
the case $L^{2}=20$, which is the first case for which our Theorem
gives two non--equivalent generalized Kummer structures. In particular
we obtain a model of $X$ as a double cover of the plane and show
that the automorphism group preserving the double cover is isomorphic
to $\ZZ_{2}\times(\ZZ_{3}\rtimes S_{3})$.

We aim to study more projective models in a forthcoming paper.

\subsection*{Acknowledgements}

The authors are grateful to the referee for his numerous comments
helping to clarify this paper. The second author is partially supported
by the ANR project No. ANR-20-CE40-0026-01 (SMAGP).

\section{The Néron--Severi lattice and its properties, Fourier--Mukai partners}

\subsection{Construction of the Néron-Severi lattice of $X$\label{subsec:Construction-of-theNS}}

Let $A$ be an abelian surface with an action of a group $G:=\ZZ/3\ZZ$
that leaves invariant each element of the space $H^{0}(A,\Omega_{A}^{2})$
(we call this action \textit{symplectic}). It is well known that the
quotient $A/G$ has $9{\bf A}_{2}$ singularities. The minimal resolution
denoted by $X:=\Km(A,G)$ is a K3 surface, called a \textit{generalized
Kummer surface}, which carries a configuration of rational curves
with Dynkin diagram $9{\bf A}_{2}$. Observe that the abelian surface
$A$ has Picard number at least $3$, see \cite[Proposition on p. 10]{Barth}
and the K3 surface $X$ has generically Picard number $19$. Let $\calK_{3}$
denotes the minimal primitive sub-lattice of the K3 lattice, $\Lambda_{K3}$,
that contains the $9$ configurations ${\bf A}_{2}$. This is a rank
$18$ negative definite even lattice of discriminant $3^{3}$, which
is described as follows. Denote by $A_{j},B_{j}$, $j=1,\ldots,9$
the nine couples of $(-2)$-curves generating the nine ${\bf A}_{2}$.
Then by \cite[Proof of Proposition 1.3]{Bertin} the lattice $\calK_{3}$
is generated by the classes $A_{1},B_{1},\ldots,A_{9},B_{9}$ and
the three classes 
\[
\begin{array}{l}
t_{1}=\frac{1}{3}(\sum_{i=1}^{9}(A_{i}-B_{i})),\\
t_{2}=\frac{1}{3}((A_{2}-B_{2})+2(A_{3}-B_{3})+A_{6}-B_{6}+2(A_{7}-B_{7})+A_{8}-B_{8}+2(A_{9}-B_{9})),\\
t_{3}=\frac{1}{3}((A_{4}-B_{4})+2(A_{5}-B_{5})+A_{6}-B_{6}+2(A_{7}-B_{7})+2(A_{8}-B_{8})+A_{9}-B_{9}),
\end{array}
\]
with intersection matrix: 
\[
\left(\begin{array}{ccc}
-6 & -6 & -6\\
-6 & -10 & -6\\
-6 & -6 & -10
\end{array}\right).
\]
The discriminant group $\calK_{3}^{\vee}/\calK_{3}$ is generated
by the classes 
\[
\begin{array}{l}
w_{1}=\frac{1}{3}(A_{5}-B_{5}+A_{7}-B_{7}+A_{8}-B_{8}),\\
w_{2}=\frac{1}{3}(2(A_{4}-B_{4})+A_{6}-B_{6}+2(A_{7}-B_{7})+A_{8}-B_{8}),\\
w_{3}=\frac{1}{3}(A_{3}-B_{3}+A_{5}-B_{5}+A_{6}-B_{6}),
\end{array}
\]
with intersection matrix: 
\[
\left(\begin{array}{ccc}
-2 & -2 & -\frac{2}{3}\\
-2 & -\frac{20}{3} & -\frac{2}{3}\\
-\frac{2}{3} & -\frac{2}{3} & -2
\end{array}\right).
\]

\begin{thm}
\label{thm:StrucureOfNeronSeveri}Assume $\rho(\Km(A,G))=19$ and
let $L$ be a generator of $\calK_{3}^{\perp}\subset\NS\Km(A,G))$
such that $L$ is big and nef. Then $L^{2}\equiv0\,\mod\,6$ or $L^{2}\equiv2\,\mod\,6$.
For an integer $k$, let us denote by $\calK_{6k}$ (respectively
$\calK_{6k+2}$) the lattice $\ZZ L\oplus\calK_{3}$ when $L^{2}=6k$
(respectively $L^{2}=6k+2$). Then 
\begin{enumerate}
\item If $L^{2}\equiv0\,\mod\,6$ then 
\[
\NS\Km(A,G))=\calK_{6k}'
\]
where $\calK_{6k}'$ is generated by $\calK_{6k}$ and by a class
$(L+v_{6k})/3$ where $v_{6k}/3\in\calK_{3}^{\vee}/\calK_{3}$ with
$L^{2}=-v_{6k}^{2}\,\mod\,18$, moreover $\calK_{6k}$ is the unique
even lattice, up to isometry, such that $[\calK_{6k}':\calK_{6k}]=3$
and $\calK_{3}$ is a primitive sublattice of $\calK_{6k}'$, so that
we can assume that 
\begin{enumerate}
\item If $L^{2}\equiv0\,\mod18$ then $v_{6k}^{2}\equiv0\,\mod18$. 
\item If $L^{2}\equiv12\,\mod18$ then $v_{6k}^{2}\equiv-12\,\mod18$. 
\item If $L^{2}\equiv6\,\mod18$ then $v_{6k}^{2}\equiv-6\,\mod18$. 
\end{enumerate}
\item If $L^{2}\equiv2\,\mod\,6$ then 
\[
\NS\Km(A,G))=\calK_{6k+2}.
\]
\end{enumerate}
\end{thm}

\begin{proof}
The fact that $L^{2}\equiv0\mod6$ or $L^{2}\equiv2\mod6$ follows
from \cite[Section 2.2]{Barth}. In the first case if $L^{2}\equiv0\mod6$
then the discriminant group of the lattice $\calK_{6k}$ contains
the generators $w_{1},w_{2},w_{3}$ and $L^{2}/6k$. Recall that for
a K3 surface the discriminant group of the Néron-Severi group is the
same as the discriminant group of the transcendental lattice, with
the quadratic form which changes the sign. Since here the rank of
the transcendental lattice is three, the number of independent generators
of the discriminant group does not exceed three, hence a class of
the discriminant group of the lattice $\calK_{6k}$ is contained in
the Néron-Severi group. This class is of the form $\frac{L+v_{6k}}{n}$,
where $\frac{v_{6k}}{n}$ belongs to the discriminant group of $\calK_{3}$.
By the previous description we necessarily have that $n=3$. Moreover
since the class $\frac{L+v_{6k}}{3}$ is now in the Néron-Severi group,
then $(\frac{L+v_{6k}}{3})^{2}\in2\ZZ$ so that $L^{2}+v_{6k}^{2}\in18\ZZ$
since $L^{2}=6k$ for some integer $k$, thus we get the cases for
$L$ and $v_{6k}$ listed in the statement. Moreover if $\frac{L+v'_{6k}}{3}$
is another class contained in the Néron-Severi group as before then
$\frac{v_{6k}-v'_{6k}}{3}$ belongs to the Néron-Severi group and
so to $\calK_{3}$ but $\frac{v_{6k}-v'_{6k}}{3}\in\calK_{3}^{\vee}/\calK_{3}$
so it must be a zero class.

For the unicity statement, we use a similar argument as in \cite[Proposition 2.2]{GarbSar},
which is as follows: First, one computes some generators of the isometry
group $O(\mathcal{K}_{3})$ of the negative definite lattice $\mathcal{K}_{3}$.
These elements act on the discriminant group $\calK_{3}^{\vee}/\calK_{3}\simeq(\ZZ/3\ZZ)^{3}$
and one obtains that the image of $O(\mathcal{K}_{3})$ in $O(\calK_{3}^{\vee}/\calK_{3})$
is a group isomorphic to $\ZZ/2\ZZ\times S_{4}$, which has exactly
$4$ orbits: $O_{z}=\{0\},O_{0},O_{1},O_{2}$, where the elements
in $O_{i}$ ($i\in\{0,1,2\}$) have square $\tfrac{2i}{3}\mod2$.
Therefore the isometry class of the gluing $\calK_{6k}'$ does not
depend on the choice of the element $v_{6k}$ and is unique for each
of the respective cases (a), (b) or (c).

In case that $L^{2}=6k+2$ with $k$ an integer, observe that a class
of the form $\frac{L+v_{6k+2}}{3}$ must satisfy $(L\cdot\frac{L+v_{6k+2}}{3})=\frac{6k+2}{3}\in\ZZ$
which is impossible, so this class does not exist and we get that
\[
\NS\Km(A,G))=\calK_{6k+2},
\]
which finishes the proof. 
\end{proof}
\begin{rem}
\label{rem:The-classes-v6t}The classes $v_{6k}$ can be chosen equal
to be $3w_{1}$ or $3w_{3}$ if $L^{2}\equiv0\mod18$; equal to $3w_{2}$
if $L^{2}\equiv6\mod18$ and equal to $3(w_{3}-w_{2})$ if $L^{2}\equiv12\mod18$.

The following result is due to Barth: 
\end{rem}

\begin{thm}
\label{thm:().-Barth-exist}(\cite[Section 2.2]{Barth}). There exists
a K3 surface $X$ such that $\NS X)=\calK_{6k}'$ (respectively $\NS X)=\calK_{6k+2}$)
for an integer $k>0$ (respectively $k\geq0$) and such a surface
$X$ is a generalized Kummer surface. 
\end{thm}

\begin{rem}
i) When the Picard number of $X$ is $19$, the discriminant group
of $\NS X)$ determines uniquely $L^{2}$ and $\NS X)$. \\
 ii) Theorems \ref{rem:The-classes-v6t} and \ref{thm:().-Barth-exist}
and their proofs (with $L\neq0$) are also valid when $A$ is a complex
non-algebraic torus, equivalently, when $L^{2}\leq0$. If $L=0$,
then $\NS X)=\calK_{3}$ is negative definite of rank $18$, and the
K3 surface $X$ is also non-algebraic. 
\end{rem}

\subsection{\label{subsec:Classes-and-polarizations}Classes and polarizations
in the Néron-Severi lattice of $X$}

\subsubsection{\label{subsec:Notations-and-divisibility}Notations and divisibility
of the classes}

Let us denote the $\cu$-curves forming the generalized Kummer configuration
by $A_{1},B_{1},\dots,A_{9},B_{9}$, where $A_{j}B_{j}=1$ and denote
by $L$ the orthogonal complement of these $18$ curves in the rank
$19$ lattice $\NS X)$, we recall that $L^{2}=2t$, for $t\in\NN^{*}$.
Let $a,\a_{j},\b_{j}\in\frac{1}{3}\ZZ$ be such that the class 
\[
\G=aL-\sum_{j=1}^{9}\left(\a_{j}A_{j}+\b_{j}B_{j}\right),
\]
is in the Néron-Severi lattice $\NS X)$. 
\begin{cor}
\label{cor:Structure-of-NS}0) For $j\in\{1,\dots,9\}$, one has $\a_{j}\in\tfrac{1}{3}\ZZ\setminus\ZZ\Leftrightarrow\b_{j}\in\tfrac{1}{3}\ZZ\setminus\ZZ$.\\
 1) Suppose $L^{2}=2\mod6$. Then $a\in\ZZ$ and if one coefficient
$\a_{j}$ or $\b_{j}$ is in $\tfrac{1}{3}\ZZ\setminus\ZZ$, then
there are $12$ or $18$ coefficients that are in $\tfrac{1}{3}\ZZ\setminus\ZZ$,\\
2) Suppose $L^{2}=0\mod6$, and $a\in\ZZ$. If one coefficient $\a_{j}$
or $\b_{j}$ is in $\tfrac{1}{3}\ZZ\setminus\ZZ$, then there are
$12$ or $18$ coefficients that are in $\tfrac{1}{3}\ZZ\setminus\ZZ$,\\
3) Suppose $L^{2}=0\mod18$ (respectively $L^{2}=6\mod18$ and $L^{2}=12\mod18$),
and $a\in\tfrac{1}{3}\ZZ\setminus\ZZ$. Then there are at least $6$
(respectively $8$ and $10$) coefficients $\a_{j}$ or $\b_{j}$
that are in $\tfrac{1}{3}\ZZ\setminus\ZZ$. \\
 4) The group $L^{\perp}/\left\langle A_{1},\dots,B_{9}\right\rangle $
(where the orthogonal of $L$ is taken in $\NS X)$) is isomorphic
to $(\ZZ/3\ZZ)^{3}$. The $27$ elements of that group are: $24$
elements which are supported on $6$ ${\bf A}_{2}$ blocs, an element
$S$ supported on the $9$ blocs ${\bf A}_{2}$, the element $2S$,
and the zero element. 
\end{cor}

\begin{proof}
It is a consequence of Theorem \ref{thm:StrucureOfNeronSeveri} and
Remark \ref{rem:The-classes-v6t}. But we may also use the more conceptual
result of Barth, that a $3$-divisible set of $A_{2}$-configuration
on a K3 surface has support on $6$ or $9$ $A_{2}$-configurations.
For example, in case 2), since $L^{2}=2\mod6$, by Theorem \ref{thm:StrucureOfNeronSeveri},
the Néron-Severi lattice is $\NS X)=\ZZ L\oplus\mathcal{K}_{3}$,
therefore the coefficient of $L$ is an integer. Let $S=\{r_{1},\dots,r_{k}\}$
be a set of $k$ elements in $\{1,\dots,9\}$ and consider the singular
surface $X_{S}$ obtained by contracting curves $A_{j},B_{j}$ to
cusps for $j\in\{r_{1},\dots,r_{k}\}$. By classical results on cyclic
covers (see e.g. \cite[Chapter I, 17]{BHPV}, see also \cite{BarthCodes}
for the particular case of cusps), there exists an order $3$ cyclic
cover branched exactly over the cusps of $X_{S}$ if and only if $\sum_{j\in R}\tfrac{1}{3}(A_{j}-B_{j})\in\NS X)$.
In \cite[Lemma 1]{Barth9cusps}, Barth obtains that when such a cover
exists, one has necessarily $k=6$ or $9$ (this is the analogous
result of a well-known result of Nikulin for the classical Kummer
surfaces). That implies that there are $12$ or $18$ coefficients
of $\G$ that are in $\tfrac{1}{3}\ZZ\setminus\ZZ$.
\end{proof}
Let $\G\in\NS X)$ be the class of a divisor and let us write 
\[
\G=aL-\frac{1}{3}\sum_{j=1}^{9}\left(a_{j}A_{j}+b_{j}B_{j}\right).
\]
with $a\in\frac{1}{3}\ZZ,\,\,a_{j},b_{j}\in\ZZ$. The intersection
numbers $\G A_{j}=\frac{1}{3}(2a_{j}-b_{j})$ and $\G B_{j}=\frac{1}{3}(2b_{j}-a_{j})$
are integers. Since $2a_{j}-b_{j}$ and $2b_{j}-a_{j}$ are divisible
by $3$, there exist integers $u_{j},v_{j}$ such that 
\[
\left\{ \begin{array}{c}
a_{j}=u_{j}+2v_{j}\\
b_{j}=2u_{j}+v_{j}
\end{array}\right.,
\]
so that we can write 
\[
\G=aL-\frac{1}{3}\sum_{j=1}^{9}\left((u_{j}+2v_{j})A_{j}+(2u_{j}+v_{j})B_{j}\right),
\]
with $u_{j},v_{j}\in\ZZ$, which is also 
\[
\G=aL-\frac{1}{3}\sum_{j=1}^{9}\left(u_{j}F_{j}+v_{j}G_{j}\right)
\]
for 
\[
F_{j}=A_{j}+2B_{j},\,G_{j}=2A_{j}+B_{j}.
\]
We have $F_{j}^{2}=G_{j}^{2}=-6,$ $F_{j}G_{j}=-3,$ so that 
\[
\G^{2}=2ta^{2}-\frac{2}{3}\sum_{j=1}^{9}\left(u_{j}^{2}+u_{j}v_{j}+v_{j}^{2}\right).
\]

Let us suppose moreover that $\G$ is the class of an irreducible
curve which is not among the $18$ curves $A_{1},\dots,B_{9}$. Then
$a\in\frac{1}{3}\NN^{*}$ and the intersection numbers $\G A_{j},\G B_{j}$
are positive or zero: 
\[
\left\{ \begin{array}{c}
2a_{j}\geq b_{j}\\
2b_{j}\geq a_{j}
\end{array}\right.,
\]
which inequalities are equivalent to 
\begin{equation}
u_{j}\geq0,\,\text{ and }v_{j}\geq0.\label{eq:ineqUiVi}
\end{equation}
.

\subsubsection{Polarizations\label{subsec:Polarizations}}

Let be $u\in\ZZ$ and define {\small{}{}{} 
\[
D=uL-\sum_{j=1}^{9}(A_{j}+B_{j}).
\]
}The $18$ curves $A_{1},B_{1},\dots,A_{9},B_{9}$ have degree $1$
for $D$: $DA_{k}=DB_{k}=1$. With the same notations, we show: 
\begin{prop}
\label{ample} The minimal integer $u_{0}$ such that for $u\geq u_{0}$
the divisor $D$ is ample is given in the following table (according
to the cases of $L^{2}$):\\
\begin{tabular}{|c|c|c|c|}
\hline 
$L^{2}=2\mod6$  & $L^{2}=2$  & $L^{2}=8$ or $14$  & $L^{2}\geq20$\tabularnewline
\hline 
$u_{0}$  & $4$  & $2$  & $1$\tabularnewline
\hline 
$L^{2}=0\mod18$  & $L^{2}=18$  & $L^{2}\geq36$  & \tabularnewline
\hline 
$u_{0}$  & $2$  & $1$  & \tabularnewline
\hline 
$L^{2}=6\mod18$  & $L^{2}=6$  & $L^{2}\geq24$  & \tabularnewline
\hline 
$u_{0}$  & $3$  & $1$  & \tabularnewline
\hline 
$L^{2}=12\mod18$  & $L^{2}=12$  & $L^{2}\geq30$  & \tabularnewline
\hline 
$u_{0}$  & $2$  & $1$  & \tabularnewline
\hline 
\end{tabular}
\end{prop}

\begin{proof}
We start by defining $D=u_{0}L-\sum_{j=1}^{9}(A_{j}+B_{j})$ with
$u_{0}$ as defined in the Table. We check that $D^{2}>0$ and $D^{\perp}$
contains no vector with square $-2$ (for example when the coefficients
on the diagonal of a Gram matrix obtained from a base of $D^{\perp}$
are multiples of $-4$). Using that the ample cone is a fundamental
domain for the Weyl group (the reflexion group generated by reflection
by vectors of square $-2$), we can choose $D$ as an ample class.
We have $DA_{j}=1$; if $A_{j}=C+C'$ with effective divisors $C,C'$,
then $DC$ or $DC'$ is $0$ thus, since $D$ is ample, $C$ or $C'$
is $0$, which proves that $A_{j}$ is a $\cu$-curve. Using that
fact, one easily check that $L$ is nef. Then for $u\geq u_{0}$,
the divisor $(u-u_{0})L+D$ is also ample, which proves the result. 
\end{proof}

\subsection{Fourier--Mukai partners and generalized Kummer structures}

Recall that for a K3 surface (resp. an abelian surface) $X$, a Fourier--Mukai
partner of $X$ is a K3 surface (resp. an abelian surface) $Y$ such
that there is an isomorphism of Hodge structures 
\[
(T(Y),\CC\o_{Y})\simeq(T(X),\CC\o_{X}),
\]
where $\o_{X}$ is a generator of $H^{0}(X,\Omega_{X}^{2})$, and
$T(X)$ is the transcendental lattice. The set of isomorphism classes
of Fourier--Mukai partners of $X$ is denoted by $\text{FM}(X)$.

Let $A$ be an abelian surface and $G_{A}$ be an order $3$ automorphism
group of $A$ acting symplectically. Let $X=\Km_{3}(A)$ (or sometimes
$\Km_{3}(A,G_{A})$ or $\Km_{3}(A,J_{A})$, where $\la J_{A}\ra=G_{A}$)
be the minimal resolution of the quotient surface $A/G_{A}$; since
$G_{A}$ is symplectic, $X$ is a generalized Kummer surface.

An isomorphism class of pairs $(B,G_{B})$ where $B$ is an abelian
surface and $G_{B}$ is an order $3$ symplectic automorphism group
such that $\Km_{3}(B)\simeq X$ is called a \textit{generalized Kummer
structure} on $X$. Let us denote by $\mathcal{K}(X)$ the set of
these isomorphism classes.

In order to state our results on a link between $\mathcal{K}(X)$
and Fourier--Mukai partners of $A$, let us recall some results and
notations in Barth's article \cite{Barth}. According to \cite[Section 1.3]{Barth},
the $G_{A}$-invariant part of $H_{2}(A,\ZZ)$ is a rank $4$ lattice
${\bf L}_{A}$, for which there is a basis $g_{1},\dots,g_{4}$ (in
the notions of \cite{Barth}, these are $g_{1}=\g_{1},g_{2}=\g_{2},g_{3}=\g_{3},g_{4}=\g_{3}+\g_{4}$)
with Gram matrix 
\[
\left(\begin{array}{cccc}
0 & 1 & 0 & 0\\
1 & 0 & 0 & 0\\
0 & 0 & 2 & 3\\
0 & 0 & 3 & 6
\end{array}\right).
\]
For $X=\Km_{3}(A)$, the lattice ${\bf L}_{X}\subset H^{2}(X,\ZZ)$
which is orthogonal to the $18$ $(-2)$-curves in $X$ has rank $4$;
it is generated by elements $\zeta_{1},\dots,\zeta_{4}$ with intersection
matrix 
\[
\left(\begin{array}{cccc}
0 & 3 & 0 & 0\\
3 & 0 & 0 & 0\\
0 & 0 & 6 & 3\\
0 & 0 & 3 & 2
\end{array}\right).
\]
The canonical rational map 
\[
\pi_{A}:A\dashrightarrow\Km_{3}(A)=X
\]
induces the morphisms 
\[
\pi_{A*}:{\bf L}_{A}\to{\bf L}_{X},\,\,\,\pi_{A}^{*}:{\bf L}_{X}\to{\bf L}_{A}
\]
which are such that 
\[
\pi_{A*}(g_{i})=\zeta_{i}\text{ for }i\leq3,\,\,\,\pi_{A*}(g_{4})=3\zeta_{4}
\]
and 
\[
\pi_{A}^{*}(\zeta_{i})=3g_{i}\text{ for }i\leq3\text{, }\pi_{A}^{*}(\zeta_{4})=g_{4},
\]
so that $\pi_{A*}\pi_{A}^{*}:{\bf L}_{X}\to{\bf L}_{X}$ and $\pi_{A}^{*}\pi_{A*}:{\bf L}_{A}\to{\bf L}_{A}$
are the multiplication by $3$ maps. The lattice $\pi_{A*}({\bf L}_{A})$
has index $3$ in ${\bf L}_{X}$, and it is easy to check that $\pi_{A*}({\bf L}_{A})$
is isometric to ${\bf L}_{A}(3)$. Here, for a lattice $L:=(L,(\,,\,))$
and a non-zero integer $m$, we define the lattice $L(m)$ by $L(m):=(L,m(\,,\,))$.

The Picard number of $A$ is either $3$ or $4$ ; we suppose that
we are in the generic case so that the Picard number is $3$. Then
the $G_{A}$-invariant part of $\NS A)$ is generated by a divisor
$L_{A}$ and the transcendental lattice of $A$ is $T(A)=L_{A}^{\perp}\subset{\bf L}_{A}$
(see \cite{Barth}). The orthogonal complement in $\NS X)$ of the
$18$ $\cu$-curves on $X$ (which are the components of the exceptional
loci of the resolution $X\to A/G_{A}$) is generated by a divisor
$L_{X}=\sum_{j=1}^{4}n_{j}\zeta_{j}\in{\bf L}_{X}$ (for some coprime
integers $n_{j}\in\ZZ$) which is such that $L_{X}^{2}=0\text{ or }2\mod6$,
moreover:\\
 $\bullet$ if $L_{X}^{2}=2\mod6$, then $\gcd(n_{4},3)=1$ and 
\[
L_{A}=\pi_{A}^{*}(L_{X})=3n_{1}g_{1}+3n_{2}g_{2}+3n_{3}g_{3}+n_{4}g_{4}
\]
$\bullet$ if $L_{X}^{2}=0\mod6$, then $n_{4}=3n_{4}'$ for some
$n_{4}'\in\ZZ$, $\gcd(n_{1},n_{2},n_{3},3)=1$ and 
\[
L_{A}=\tfrac{1}{3}\pi_{A}^{*}(L_{X})=n_{1}g_{1}+n_{2}g_{2}+n_{3}g_{3}+n_{4}'g_{4}.
\]
The orthogonal complement of $L_{X}$ and the $18$ $\cu$-curves
on $X$ is the transcendental lattice $T(X)$ of $X=\Km_{3}(A)$;
it is contained in ${\bf L}_{X}$. One has $\pi_{A}^{*}L_{X}=\nu L_{A}$
with $\nu\in\{1,3\}$. 
\begin{thm}
\label{thm:Fourier-Mukai}Let $(A,G_{A})$ and $(B,G_{B})$ be two
abelian surfaces with an order $3$ symplectic automorphism group
and let $X=\Km_{3}(A)$. Suppose that $X$ has Picard number $19$
and $L_{X}^{2}=2\mod6$. We have $\Km_{3}(B)\simeq\Km_{3}(A)=X$ (i.e.
$\{(B,G_{B})\}\in\mathcal{K}(X)$) if and only if $B$ is a Fourier--Mukai
partner of $A$. 
\end{thm}

The proof is similar to \cite[Theorem 0.1]{HLOY} for the classical
Kummer surfaces. Let us prove the following result 
\begin{lem}
\label{lem:The-canonical-rational}Suppose that $L_{X}^{2}=2\mod6$.
The canonical rational map 
\[
\pi_{A}:A\dashrightarrow X=\Km_{3}(A)
\]
induces a Hodge isometry 
\[
\pi_{A*}:(T(A)(3),\CC\o_{A})\stackrel{}{\to}(T{}_{\Km_{3}(A)},\CC\o_{\Km_{3}(A)}).
\]
\end{lem}

For the proof of Lemma \ref{lem:The-canonical-rational}, we will
use the following elementary result: 
\begin{lem}
\label{fact:index}Let $N$ be a lattice and let $N'\subset N$ be
a sub-lattice of finite index $k=[N:N']$, let $M\subset N$ be a
sub-lattice and let $M'=M\cap N'$. The index $[M:M']$ divides $[N:N']$. 
\end{lem}

\begin{proof}
The map $\phi:m+M'\in M/M'\to m+N'\in N/N'$ is a well defined morphism.
For $m\in M$, one has $\phi(m+M')=N'\in N/N'$ if and only if $m\in N'$
thus if and only if $m\in N'\cap M=M'$: the kernel of $\phi$ is
trivial. Thus the morphism $\phi$ is injective and identifying $M/M'$
with its image, by Lagrange Theorem the order of $M/M'$ divides the
order of $N/N'$. 
\end{proof}
\begin{proof}
(Of Lemma \ref{lem:The-canonical-rational}). Since $\pi_{A*}$ is
an isometry between ${\bf L}_{A}(3)$ and $\pi_{A*}({\bf L}_{A})$,
it restricts to an isometry from $T(A)(3)$ to $\pi_{A*}(T(A))$ (let
us remark that one may also prove that $\pi_{A*}(T(A))$ is isometric
to $T(A)(3)$ by using \cite[Proposition 1.1 \& Remark]{Inose} of
Inose).

Define $N={\bf L}_{X}$, $N'=\pi_{A*}({\bf L}_{A})$, $v=\pi_{A*}(L_{A})$,
$M=v^{\perp_{N}}=L_{X}^{\perp}(=T(X))$ and 
\[
M'=\pi_{A*}(L_{A}^{\perp})=\pi_{A*}(T(A))=M\cap N'.
\]
Since $[N:N']=3$, by Lemma \ref{fact:index}, we get that $[M:M']\in\{1,3\}$.

Suppose that $L_{X}^{2}=2\mod6$ and write 
\[
L_{A}=\pi_{A}^{*}(L_{X})=3n_{1}g_{1}+3n_{2}g_{2}+3n_{3}g_{3}+n_{4}g_{4},
\]
(we recall that $L_{X}=n_{1}\zeta_{1}+n_{2}\zeta_{2}+n_{3}\zeta_{3}+n_{4}\zeta_{4}$
with coprime integers $n_{1},\dots,n_{4}$). Then $v=3w$ for $w=L_{X}$.
Again by Lemma \ref{fact:index}, the lattice 
\[
M'\oplus\ZZ v=(M\oplus\ZZ w)\cap N',
\]
has index $1$ or $3$ in $M\oplus\ZZ w$. Since 
\[
[M\oplus\ZZ w:M'\oplus\ZZ v]=[M:M'][\ZZ w:\ZZ v]=3[M:M'],
\]
that forces $[M:M']=1$, which implies that 
\[
\pi_{A*}(T(A))=\pi_{A*}(L_{A}{}^{\perp})=L_{X}^{\perp}=T(X),
\]
and therefore $T(X)$ is isometric to $T(A)(3)$, since $\pi_{A*}(T(A))\simeq T(A)(3)$. 
\end{proof}
\begin{rem}
For any $k\in\ZZ$, the polarization $L_{X}=\zeta_{1}+k\zeta_{2}$
is such that $L_{X}^{2}=6k$ (as remarked by Barth in \cite{Barth}).
Using the Gram matrix of the $\zeta_{k}$'s, one finds that the transcendental
lattice of $X$ is generated by elements $\zeta_{1}-k\zeta_{2},\zeta_{3},\zeta_{4}$,
whereas $T(A)$ is generated by $g_{1}-kg_{2},g_{3},g_{4}$. Thus,
for these examples one obtains that $\pi_{A*}(T(A))$ (generated by
$\zeta_{1}-k\zeta_{2},\zeta_{3},3\zeta_{4}$) has index $3$ in $T(X)$.
In fact, using Lemma \ref{fact:index}, one can prove more generally
that for any polarization $L_{X}$ such that $L_{X}^{2}=6k$, the
lattice $\pi_{A*}(T(A))$ has index $3$ in $T(X)$. 
\end{rem}

Let us prove Theorem \ref{thm:Fourier-Mukai}. 
\begin{proof}
Suppose that $L_{X}^{2}=2\mod6$. From Lemma \ref{lem:The-canonical-rational},
there exists an isomorphism of Hodge structures 
\[
(T(B),\CC\o_{B})\simeq(T(A),\CC\o_{A})
\]
if and only if there is an isomorphism of Hodge structures 
\[
(T(\Km_{3}(B)),\CC\o_{\Km_{3}(B)})\simeq(T(\Km_{3}(A)),\CC\o_{\Km_{3}(A)}).
\]
Therefore $B$ is a Fourier--Mukai partner of $A$ if and only if
$\Km_{3}(B)$ is a Fourier--Mukai partner of $\Km_{3}(A)$. By \cite[Corollary 2.6]{HLOY2},
since $X=\Km_{3}(A)$ has Picard number $19>2+\ell$ (here $\ell$
is the length of the discriminant group of $\NS X)$, which is also
the length of $T(X)$, and therefore is $\leq3$), one has $\{\Km_{3}(A)\}=\text{FM}(\Km_{3}(A))$
(the isomorphism class of $\Km_{3}(A)$ is the unique Fourier--Mukai
partner of $\Km_{3}(A)$), therefore $B$ is a Fourier--Mukai partner
of $A$ if and only if $\Km_{3}(B)\simeq\Km_{3}(A)=X$. In other words:
$B$ is a Fourier--Mukai partner of $A$ if and only if $\Km_{3}(B)\in\mathcal{K}(X)$. 
\end{proof}
By \cite[Proposition 5.3]{BriMac}, the number $|\text{FM}(A)|$ of
Fourier--Mukai partners of $A$ is finite. So, in case $L_{X}^{2}=2\mod6$,
to find the number of generalized Kummer structures on $X$ reduces
to find the number of conjugacy classes of order $3$ symplectic groups
$G'$ on a finite number of abelian surfaces $B$ such that $\Km_{3}(B,G')\simeq\Km_{3}(A)$.

\section{Constructions of Kummer structures}

\subsection{$\protect\cu$-classes with intersection one with $A_{1}$\label{subsec:-classes with intersection-1-with-A_=00003D00003D00007B1=00003D00003D00007D}}

Let $A_{1},B_{1},\dots,A_{9},B_{9}$ be a generalized Nikulin configuration
on a generalized Kummer surface $X$. Let $L$ be the big and nef
generator of the orthogonal complement of these $18$ curves and let
$t\in\NN$ such that $L^{2}=2t$. According to the two possible cases,
$L^{2}=2\text{\,mod }6$ and $L^{2}=0\mod6$, we denote by $k\in\NN$
the integer such that $t=1+3k$ in the first case, and such that $t=3k$
in the second case.

Our aim is to construct a $(-2)$-curve $B_{1}'\neq B_{1}$ in the
lattice $\mathbb{L}$ generated by $L,A_{1},B_{1}$ such that $A_{1}B_{1}'=1$,
so that $(A_{1},B_{1}')$ is another ${\bf A}_{2}$-configuration.
As we will see in the next Section, under some conditions on $t,$
such a curve exists and is unique. In order to prove this result,
we will study the properties of $L'$, the orthogonal complement of
$A_{1},B_{1}'$ in $\mathbb{L}$.

By Theorem \ref{thm:StrucureOfNeronSeveri}, an element in the lattice
generated by $L,A_{1},B_{1}$ has the form 
\[
B_{1}'=aL-(a_{1}A_{1}+b_{1}B_{1}),
\]
for integers $a,a_{1},b_{1}$. The class $B_{1}'$ satisfies $B'_{1}A_{1}=1$
if and only if $2a_{1}-b_{1}=1,$ which gives 
\[
B_{1}'=aL-(a_{1}A_{1}+(2a_{1}-1)B_{1}).
\]
Moreover since we search for a $\cu$-curve, we have 
\[
B_{1}'^{2}=-2=2ta^{2}-6a_{1}^{2}+6a_{1}-2,
\]
which is equivalent to 
\[
3a_{1}(a_{1}-1)-ta^{2}=0.
\]
We can write that condition as 
\[
3\left((2a_{1}-1)^{2}-1\right)-4ta^{2}=0,
\]
which is equivalent to 
\begin{equation}
3(2a_{1}-1)^{2}-4ta^{2}=3.\label{eq:TT1}
\end{equation}

\subsubsection{\label{subsec:const-B-L-form-Teq1mod3}Case $t=1\protect\mod3$}

Since we suppose that $t=1\mod3$, the integer $a$ in Equation \eqref{eq:TT1}
must be divisible by $3$. Let us define the integers $x_{0},y_{0}$
by 
\[
x_{0}=2a_{1}-1,\,\,a=3y_{0}.
\]
Equation \eqref{eq:TT1} is then equivalent to the Pell-Fermat equation
\begin{equation}
x_{0}^{2}-12ty_{0}^{2}=1.\label{eq:Pell-Fermat-12t}
\end{equation}
Since $t=1\mod3$, the integer $12t$ is never a square and there
exist non-trivial solutions. Let us fix such a solution $(x_{0},y_{0})$
(we observe that $x_{0}$ is necessarily odd). Then 
\[
B_{1}'=3y_{0}L-(a_{1}A_{1}+(2a_{1}-1)B_{1})
\]
with $2a_{1}-1=x_{0}$ is such that $B_{1}'^{2}=-2$ and $B_{1}'A_{1}=1$,
and conversely, all $\cu$-classes with these two properties are obtained
in that way.

Let us search for the class of $L'=\a L-(\b_{1}A_{1}+\b_{1}'B_{1})$,
$\a,\b_{1},\b_{1}'\in\ZZ$, such that $L'$ generates the orthogonal
complement of $A_{1},B_{1}',A_{2},B_{2},\dots,A_{9},B_{9}$. From
$L'A_{1}=0$, we obtain 
\[
\b_{1}'=2\b_{1}.
\]
Since also $L'B_{1}'=0$, we get 
\[
2\a at=3\b_{1}(2a_{1}-1),
\]
in other words: 
\[
2\a ty_{0}=\b_{1}x_{0}.
\]
Since by equation \ref{eq:Pell-Fermat-12t} the integers $x_{0},y_{0}$
are co-primes and $x_{0}$ is co-prime to $2t$, we get: 
\[
\b_{1}=2ty_{0},\,\,\a=x_{0},
\]
and the class $L'=x_{0}L-2ty_{0}(A_{1}+2B_{1})$ is primitive with
$L'L>0$, and such that 
\[
L'^{2}=2tx_{0}^{2}-24t^{2}y_{0}^{2}=2t.
\]
For any solution $(x_{0},y_{0})$ of the Pell-Fermat equation \eqref{eq:Pell-Fermat-12t},
the classes 
\[
\begin{array}{l}
B_{1}'=3y_{0}L+(\tfrac{1}{2}(x_{0}+1)A_{1}+x_{0}B_{1})\\
L'=x_{0}L-2ty_{0}(A_{1}+2B_{1})
\end{array}
\]
in $\mathbb{L}$ have the properties $B_{1}'^{2}=-2,\,B_{1}'A_{1}=1,\,L'A_{1}=L'B_{1}'=0$,
$L'^{2}=2t,\,L'L>0$ and all the classes $B,L_{0}$ with these properties
are obtained in that way.

\subsubsection{\label{subsec:Constr-B-L-form-teq0mod3}Case $t=0\protect\mod3$}

Let us suppose that $t=0\mod3$, and let $k\in\NN^{*}$ such that
$t=3k$. Then the equation 
\[
3(2a_{1}-1)^{2}-4ta^{2}=3
\]
is equivalent to 
\[
(2a_{1}-1)^{2}-4ka^{2}=1.
\]
By defining $x_{0}=2a_{1}-1,\,\,y_{0}=a$, we are reduced to the Pell-Fermat
equation 
\begin{equation}
x_{0}^{2}-4ky_{0}^{2}=1,\label{eq:Pell-Fermat-4k}
\end{equation}
($x_{0}$ is necessarily odd). Let us fix such a solution $(x_{0},y_{0})$.
Then 
\[
B_{1}'=y_{0}L-(a_{1}A_{1}+(2a_{1}-1)B_{1}),
\]
with $2a_{1}-1=x_{0}$ satisfies $B_{1}'^{2}=-2,\,B_{1}'A_{1}=1$.

Let us search for the class of $L'=\a L-(\b_{1}A_{1}+\b_{1}'B_{1})$,
$\a,\b_{1},\b_{1}'\in\ZZ$, such that $L'$ generates the orthogonal
complement of $A_{1},B_{1}',A_{2},B_{2},\dots,A_{9},B_{9}$. From
$L'A_{1}=0$, we obtain 
\[
\b_{1}'=2\b_{1},
\]
and since $L'B_{1}'=0$, we get 
\[
2\a at=3\b_{1}(2a_{1}-1),
\]
in other words: 
\[
2\a ky_{0}=\b_{1}x_{0}.
\]
Since $x_{0},y_{0}$ are co-primes, and $L'$ is primitive, we get:
\[
\b_{1}=2ky_{0},\,\,\a=x_{0}
\]
and $L'=x_{0}L-2ky_{0}(A_{1}+2B_{1}),$ thus 
\[
L'^{2}=2tx_{0}^{2}-24k^{2}y_{0}^{2}=2t(x_{0}^{2}-4ky_{0}^{2})=2t.
\]
For any solution $(x_{0},y_{0})$ of the Pell-Fermat equation \eqref{eq:Pell-Fermat-4k},
the classes 
\[
\begin{array}{l}
B_{1}'=y_{0}L+(\tfrac{1}{2}(x_{0}+1)A_{1}+x_{0}B_{1})\\
L'=x_{0}L-2ky_{0}(A_{1}+2B_{1})
\end{array}
\]
have the properties we required: $B_{1}'^{2}=-2,\,B_{1}'A_{1}=1,\,L'A_{1}=L'B_{1}'=0$
and all classes with these properties are obtained in that way.

%\pagebreak

\subsection{\label{subsec:Exisentence-and-unicity-of-B1}Existence and unicity
of $B_{1}'$}

\subsubsection{\label{subsc:B1-Irred-Cas-t=00003D00003D00003D1mod3}Case $t=1\protect\mod3$}

Let $(x_{0},y_{0})$ be a non-trivial solution of the Pell-Fermat
equation \eqref{eq:Pell-Fermat-12t}. In Section \ref{subsec:const-B-L-form-Teq1mod3},
we defined the classes 
\[
\begin{array}{l}
L'=x_{0}L-2ty_{0}(A_{1}+2B_{1}),\\
B_{1}'=3y_{0}L-(\frac{1}{2}(x_{0}+1)A_{1}+x_{0}B_{1}).
\end{array}
\]
Let us prove the following result: 
\begin{prop}
\label{prop:L-is-Big-and-Nef}Suppose that $(x_{0},y_{0})$ is the
fundamental solution of the Pell-Fermat equation \eqref{eq:Pell-Fermat-12t}.
The class $L'$ is nef. Suppose that $L'\G=0$ for the class $\G$
of an irreducible $\cu$-curve. Then $\G$ is in $\{A_{1},B_{1}',A_{2},\dots,A_{9},B_{9}\}$. 
\end{prop}

\begin{proof}
Let $\G$ be a $\cu$-curve, we write it as (see Section \ref{subsec:Notations-and-divisibility})
\[
\G=aL-\frac{1}{3}\sum_{j=1}^{9}\left((u_{j}+2v_{j})A_{j}+(2u_{j}+v_{j})B_{j}\right),
\]
with $u_{j}\geq0,\,v_{j}\geq0$ (we can apply inequality \eqref{eq:ineqUiVi}
since we suppose that $\G$ is irreducible) and $u_{j},v_{j}\in\ZZ$,
$a\in\ZZ$, so that 
\begin{equation}
\G^{2}=2ta^{2}-\frac{2}{3}\sum_{j=1}^{9}\left(u_{j}^{2}+u_{j}v_{j}+v_{j}^{2}\right)=-2.\label{eq:etik1}
\end{equation}
Suppose that $\G L'\leq0$. This is equivalent to: 
\[
\left(aL-\frac{1}{3}\left((u_{1}+2v_{1})A_{1}+(2u_{1}+v_{1})B_{1}\right)\right)\left(x_{0}L-2ty_{0}(A_{1}+2B_{1})\right)\leq0
\]
which is equivalent to 
\[
ax_{0}\leq(2u_{1}+v_{1})y_{0}.
\]
By taking the square (recall that $x_{0}>0,y_{0}>0$ and $a\geq0$),
we get 
\[
x_{0}^{2}a^{2}\leq\left(4(u_{1}^{2}+u_{1}v_{1}+v_{1}^{2})-3v_{1}^{2}\right)y_{0}^{2},
\]
which is equivalent to 
\[
-4(u_{1}^{2}+u_{1}v_{1}+v_{1}^{2})y_{0}^{2}\leq-x_{0}^{2}a^{2}-3v_{1}^{2}y_{0}^{2},
\]
thus 
\[
-\frac{2}{3}(u_{1}^{2}+u_{1}v_{1}+v_{1}^{2})\leq-\frac{1}{6y_{0}^{2}}(x_{0}^{2}a^{2}+3v_{1}^{2}y_{0}^{2})
\]
and using equality $\G^{2}=-2$ in \eqref{eq:etik1}, we get 
\[
-2\leq2ta^{2}-\frac{1}{6y_{0}^{2}}(x_{0}^{2}a^{2}+3v_{1}^{2}y_{0}^{2})-\frac{2}{3}S,
\]
here and hereafter, we denote 
\[
S=\sum_{j=2}^{9}\left(u_{j}^{2}+u_{j}v_{j}+v_{j}^{2}\right)\geq0.
\]
Thus we obtain 
\[
\tfrac{1}{2}v_{1}^{2}+\tfrac{1}{6}a^{2}(\frac{x_{0}^{2}}{y_{0}^{2}}-12t)+\tfrac{2}{3}S\leq2,
\]
which is equivalent to 
\begin{equation}
\tfrac{1}{2}v_{1}^{2}+\tfrac{1}{6}\frac{a^{2}}{y_{0}^{2}}+\tfrac{2}{3}S\leq2.\label{eq:importante-pour-L-nef}
\end{equation}
If $a=0$, then $\G=A_{1}$ or $A_{j},B_{j}$ with $j\geq2$. We therefore
suppose that $a\neq0$, and then $a>0$ since $\G$ is effective.
Let us suppose that one of the coefficients $\a_{j}=\frac{1}{3}(u_{j}+2v_{j}),\,\b_{j}=\frac{1}{3}(2u_{j}+v_{j})$
of $\G$ is in $\frac{1}{3}\ZZ\setminus\ZZ$. Then by Corollary \ref{cor:Structure-of-NS},
at least $12$ of the coefficients $\a_{j},\b_{j}$ are non-zero,
and therefore at least $5$ of the $u_{j}$ or $v_{j}$ (which are
$\geq0$) with the condition $j\geq2$ are non-zero. That implies
that $S\geq5$, thus $\frac{2}{3}S\geq\frac{10}{3}>2$, which is a
contradiction. So all the coefficients of $\G$ are integers. Suppose
that $S>0$. Since $S<3$, there exist one or two indices $j\geq2$
such that $u_{j}$ or $v_{j}$ is equal to $1$ (and the other coefficients
with index $k\geq2$ are $0$). But then the coefficients of $\G$
are not integral: the only possibility is $S=0$. Since $a\neq0$,
we have $v_{1}\in\{0,1\}.$ Since $\G^{2}=-2$, the integers $a,u_{1},v_{1},t$
satisfies 
\[
2ta^{2}-\frac{2}{3}\left(u_{1}^{2}+u_{1}v_{1}+v_{1}^{2}\right)=-2,
\]
which is equivalent to 
\[
u_{1}^{2}+u_{1}v_{1}+v_{1}^{2}=3(ta^{2}+1).
\]
If $v_{1}=0$, since $t=1\mod3$, $ta^{2}+1=1$ or $2\mod3$. But
then $3(ta^{2}+1)$ is not a square, and there is no such a solution.
Therefore $v_{1}=1$, and from inequality \eqref{eq:importante-pour-L-nef},
the integer $a$ is in the range 
\[
1\leq a\leq3y_{0}.
\]
The equality 
\[
u_{1}^{2}+u_{1}+1=3(ta^{2}+1),
\]
is equivalent to 
\[
(2u_{1}+1)^{2}+3=12(ta^{2}+1),
\]
thus to 
\[
(2u_{1}+1)^{2}=3(4ta^{2}+3).
\]
Then $2u_{1}+1$ must be divisible by $3$: let $w$ be such that
$3w=2u_{1}+1$. The above equation is equivalent to 
\[
3w^{2}=4ta^{2}+3,
\]
which in turn, since $t=1\text{ mod 3}$, implies that there exists
an integer $A$ such that $a=3A$, then the equation is equivalent
to the Pell-Fermat equation 
\[
w^{2}-12tA^{2}=1.
\]
Let $w_{0},A_{0}$ be a solution of that Pell-Fermat equation, then
$a=3A_{0},u_{1}=\frac{1}{2}(3w_{0}-1)$, $v_{1}=1$ are such that
$\G^{2}=-2$. Since $a\leq3y_{0}$, and we now use that $(x_{0},y_{0})$
is the fundamental solution of the Pell-Fermat equation to conclude
that $a=3y_{0}$ and $u_{1}=\frac{1}{2}(3x_{0}-1),$ for which integers
one has $\Gamma=B_{1}'$, and $L'B_{1}'=0$. That concludes the proof. 
\end{proof}
We obtain: 
\begin{cor}
\label{cor:B1'-Is-(-2)-curve}The divisor $B_{1}'=3y_{0}L-(\frac{1}{2}(x_{0}+1)A_{1}+x_{0}B_{1})$
is the class of a $\cu$-curve.\\
 The curves $B_{1}$ and $B_{1}'$ are the unique $\cu$-curves in
the lattice generated by $L,A_{1},B_{1}$ which have intersection
$1$ with $A_{1}$. 
\end{cor}

\begin{proof}
A suitable multiple of the big and nef divisor $L'$ defines a map
which is is generically one to one onto a projective model. It contracts
the irreducible curves subjacent to $B_{1}'$ and the curves $A_{1},A_{2},B_{2},\dots,A_{9},B_{9}$
to ADE singularities. By the genericity assumption on the K3 surface
$X$, the surface has Picard number $19$, which forces the image
of $A_{1}+B_{1}'$ to be a ${\bf A}_{2}$-singularity, therefore the
curve $B_{1}'$ is irreducible.

Let us prove the unicity claim. By Section \ref{subsec:-classes with intersection-1-with-A_=00003D00003D00007B1=00003D00003D00007D},
if $\tilde{B}_{1}$ is the class of a $\cu$-curve such that $\tilde{B}_{1}A_{1}=1$,
there exists a solution $(x,y)$ of the Pell-Fermat equation \eqref{eq:Pell-Fermat-12t}
such that 
\[
\tilde{B}_{1}=3yL-(\frac{1}{2}(x+1)A_{1}+xB_{1}).
\]
Suppose $\tilde{B}_{1}\neq B_{1}$ ie $(x,y)\neq(-1,0)$. Then since
$\tilde{B}_{1}$ is effective, one has $x>0,y>0$ and there exists
$k\in\NN^{*}$ such that $x=x_{k},\,\,y=y_{k}$ for 
\[
x_{k}+\sqrt{12t}y_{k}=(x_{0}+\sqrt{12t}y_{0})^{k}=(x_{0}+\sqrt{12t}y_{0})(x_{k-1}+\sqrt{12t}y_{k-1}),
\]
where $(x_{0},y_{0})$ is the fundamental solution of the Pell-Fermat
equation \eqref{eq:Pell-Fermat-12t} (see e.g. \cite{Barbeau}, Section
4.2). One has 
\[
\begin{array}{cc}
\tilde{B}_{1}B_{1}' & =(3yL-(\frac{1}{2}(x+1)A_{1}+xB_{1}))(3y_{0}L-(\frac{1}{2}(x_{0}+1)A_{1}+x_{0}B_{1}))\\
 & =18yy_{0}t-\frac{3}{2}xx_{0}-\frac{1}{2}
\end{array},
\]
therefore $\tilde{B}_{1}B_{1}'<0$ if and only if $xx_{0}+\frac{1}{3}>12yy_{0}$.
Using an induction on $k$, one can check that this is the case for
all $k\geq1$. Thus if $k>1$, the $\cu$-class $\tilde{B}_{1}$ cannot
be the class of an irreducible curve. We observe moreover that if
$k=1$, then $\tilde{B}_{1}=B_{1}'$, and that concludes the proof. 
\end{proof}

\subsubsection{\label{subsc:B1-Irred-t=00003D00003D00003D0mod3}Case $t=0\protect\mod3$}

Suppose that $L^{2}=2t=6k$. Let $(x_{0},y_{0})$ be the fundamental
solution of the Pell-Fermat equation \eqref{eq:Pell-Fermat-4k}. Let
us search when the $\cu$-class 
\[
B_{1}'=y_{0}L-(\tfrac{1}{2}(x_{0}+1)A_{1}+x_{0}B_{1}),
\]
is the class of a $\cu$-curve. Recall that 
\[
L'=x_{0}L-2ky_{0}(A_{1}+2B_{1}),
\]
generates in $\NS X)$ the orthogonal complement of $A_{1},B_{1}',A_{2},B_{2},\dots,A_{9},B_{9}$. 
\begin{prop}
\label{prop:L-is-Big-and-Nef-4k}a) Suppose that $L^{2}=6\text{ or }12\mod18$
or $L^{2}=0\mod18$ and $3|y_{0}$. Then the class $L'$ is nef. Moreover
suppose that $\G$ is the class of an irreducible $(-2)$-curve such
that $L'\G=0$, then $\G\in\{A_{1},B_{1}',A_{2},\dots,A_{9},B_{9}\}$.
\\
b) Suppose that $L^{2}=0\mod18$ and $3\not|y_{0}$. Up to exchanging
the curves $A_{1}$ and $B_{1}$, the same result holds true. 
\end{prop}

\begin{rem}
i) If $\Gamma$ is in $\{A_{1},B_{1}',A_{2},\dots,A_{9},B_{9}\}$,
one has $L'\G=0$. The classes $D=A_{j}+B_{j}$ (for $j\geq2$) or
$D=A_{1}+B_{1}'$ are also a $\cu$-classes such that $L'D=0$.\\
ii) The class $B_{1}'$ is constructed such that $A_{1}B_{1}'=1$.
Using the same solutions of the Pell-Fermat equation, we can also
construct a class $A_{1}'$ such that $A_{1}'B_{1}=1$, this is what
we mean in b), when we say that the result holds true up to exchanging
the role of the curves $A_{1}$ and $B_{1}$.
\end{rem}

\textit{Proof} (of Proposition \ref{prop:L-is-Big-and-Nef}). As we
will see the proof of Proposition \ref{prop:L-is-Big-and-Nef} is
easy for the cases $L^{2}=6,12\mod18$, the main difficulty is for
$L^{2}=0\mod18$. 

By definition, $L$ is nef if and only if for any $\cu$-curve $\G$,
one has $L\G\geq0$. Let therefore 
\[
\G=aL-\frac{1}{3}\sum_{j=1}^{9}\left((u_{j}+2v_{j})A_{j}+(2u_{j}+v_{j})B_{j}\right),
\]
be an effective $\cu$-class and let us suppose moreover that this
is the class of an irreducible $\cu$-curve not in $\{A_{1},\dots,B_{9}\}$.
Then by inequality \eqref{eq:ineqUiVi}, necessarily one has $u_{j}\geq0,\,v_{j}\geq0$,
$u_{j},v_{j}\in\ZZ$, and $a>0$, in $\frac{1}{3}\ZZ$ (since $t=0\mod3$).
As in Section \ref{subsc:B1-Irred-Cas-t=00003D00003D00003D1mod3},
let us study if $L'$ is nef. After computations similar to those
of Section \ref{subsc:B1-Irred-Cas-t=00003D00003D00003D1mod3}, we
obtain that $\G L'\leq0$ if and only if 
\begin{equation}
\frac{1}{2}v_{1}^{2}+\frac{3a^{2}}{2y_{0}^{2}}+\frac{2}{3}S\leq2,\label{eq:Important-K-L2eq6k}
\end{equation}
where $S=\sum_{j=2}^{9}\left(u_{j}^{2}+u_{j}v_{j}+v_{j}^{2}\right)\in\NN.$
We observe that $S<3$ since $a\neq0$. 

\textbf{Case 1: one of the coefficients of $\G$ is in $\tfrac{1}{3}\ZZ\setminus\ZZ$.}
Suppose that one of the coefficients of $\G$ is in $\frac{1}{3}\ZZ\setminus\ZZ$.
By Corollary \ref{cor:Structure-of-NS}, there are at least $6,8$
or $10$ coefficients $\a_{j}=\frac{1}{3}(u_{j}+2v_{j}),\b_{j}=\frac{1}{3}(2u_{j}+v_{j})$
that are in $\frac{1}{3}\ZZ\setminus\ZZ$ according if $L^{2}=0,\,6\text{ or }12\mod18$.
Thus respectively, at least $3,4$ or $5$ integers $u_{j},v_{j}$
are non-zero, and therefore the sum $S$ (which is over the indices
$j\geq2$) is larger or equal to $2,3$ or $4$ respectively. Since
$S<3$, Corollary \ref{cor:Structure-of-NS} implies that $L^{2}=0\mod18$,
$S=2$, moreover $a\in\frac{1}{3}\ZZ\setminus\ZZ$ (still by Corollary
\ref{cor:Structure-of-NS}) and Equation \eqref{eq:Important-K-L2eq6k}
is then equivalent to
\begin{equation}
\frac{1}{2}v_{1}^{2}+\frac{3a^{2}}{2y_{0}^{2}}\leq\frac{2}{3}.\label{eq:Import2}
\end{equation}
\textbf{Sub-case 1) a) $v_{1}=1$ (and }$L^{2}=0\mod18$\textbf{).
}Let us suppose that $v_{1}=1$ and define $a'\in\ZZ$ such that $a=\frac{a'}{3}$
(since $a\in\frac{1}{3}\ZZ\setminus\ZZ$, $a'$ is coprime to $3$).
Inequality \eqref{eq:Import2} implies that $a'\leq y_{0}$. Since
$S=2$ and $v_{1}=1$, the $\cu$-class $\G$ has the form: 
\[
\G=\tfrac{1}{3}a'L-\tfrac{1}{3}(u_{1}F_{1}+G_{1}+H+H')
\]
with classes $H,H'$ in the set $\{F_{j},G_{j}\,|\,j\geq2\}$ and
such that $HH'=0$ (we recall that $F_{j}=A_{j}+2B_{j},$ $G_{j}=2A_{j}+B_{j}$).
Equality $\G^{2}=-2$ is equivalent to 
\[
\frac{2}{3}ka'^{2}-\frac{2}{3}(u_{1}^{2}+u_{1}+1)-\frac{2}{3}-\frac{2}{3}=-2,
\]
which is equivalent to 
\[
(u_{1}^{2}+u_{1}+1)-ka'^{2}=1,
\]
(observe here that since $L^{2}=0\mod18$, one has $k=0\mod3$, thus
$u_{1}^{2}+u_{1}=0\mod3$) and finally, to 
\[
(2u_{1}+1)^{2}-4ka'^{2}=1.
\]
Since $a'\leq y_{0}$ and $(x_{0},y_{0})$ is the fundamental solution
of the Pell-Fermat equation \eqref{eq:Pell-Fermat-4k}, we have $a'=y_{0}$,
(in other words $a=\frac{1}{3}y_{0}$) and $u_{1}=\frac{1}{2}(x_{0}-1)$.
We obtain that the class $\G$ is 
\begin{equation}
\G_{0}=\tfrac{1}{3}y_{0}L-\tfrac{1}{3}\left(\tfrac{1}{2}(x_{0}-1)F_{1}+G_{1}+H+H'\right).\label{eq:degGamma_0}
\end{equation}
By Corollary \ref{cor:Structure-of-NS}, this class can be in the
Néron-Severi group only if $y_{0}$ is coprime to $3$. Now, let us
suppose that $y_{0}$ is coprime to $3$ (recall that we are in the
case $L^{2}=0\mod18$). We will see that up to exchanging the role
of $A_{1}$ and $B_{1}$, one can also suppose that such a $\G_{0}$
is not in the Néron-Severi group.

If the pair $\{H,H'\}$ in the definition of $\G_{0}$ of equation
\ref{eq:degGamma_0} exists, it is unique, otherwise the difference
between the two obtained $(-2)$-classes $\G_{0}$ would have at least
$2$ and at most $8$ coefficients in $\tfrac{1}{3}\ZZ\setminus\ZZ$,
with the coefficient of $L$ in $\ZZ$, but this is impossible by
part 2) of Corollary \ref{cor:Structure-of-NS}.

For some classes $H_{1},H_{1}'\in\{F_{j},G_{j}\,|\,j\geq2\}$ with
$H_{1}H_{1}'=0$, let us consider the $\cu$-class 
\[
\G_{0}'=\tfrac{1}{3}y_{0}L-\tfrac{1}{3}\left(F_{1}+\tfrac{1}{2}(x_{0}-1)G_{1}+H_{1}+H'_{1}\right)
\]
The number of coefficients in $\tfrac{1}{3}\ZZ\setminus\ZZ$ of the
class
\[
\G_{0}-\G_{0}'=\tfrac{1}{3}((u_{1}-1)(G_{1}-F_{1})+H+H'-H_{1}-H_{1}')
\]
is at least $2$ (because $u_{1}\neq1\mod3$) and at most $10$ (because
each class $H,H',H_{1},H_{1}'$ can give at most two coefficients
in $\tfrac{1}{3}\ZZ\setminus\ZZ$). By part 2) of Corollary \ref{cor:Structure-of-NS},
this is impossible. We thus obtain that either $\G_{0}$ or $\G_{0}'$
is not in the Néron-Severi group. Thus up to exchanging $A_{1}$ and
$B_{1}$, and working with $B_{1}$ instead of $A_{1}$, one can suppose
that $\G_{0}$ is not in the Néron-Severi group. 
\begin{rem}
If $\G_{0}\in\NS X)$, then 
\[
B_{1}'=3\G_{0}+B_{1}+H+H',
\]
and the curve $B_{1}'$ is not irreducible. 
\end{rem}

\textbf{Sub-case 1) b) $v_{1}=0$ (and }$L^{2}=0\mod18$\textbf{).}
We are still in the case $L^{2}=0\mod18$, $S=2$, $a\in\tfrac{1}{3}\ZZ\setminus\ZZ$
and we suppose now that $v_{1}=0$. Equation \ref{eq:Import2} is
equivalent to $a'\leq2y_{0},$ where $a=\frac{a'}{3}$. The $\cu$-class
$\G$ has the form: 
\[
\G=\tfrac{1}{3}a'L-\tfrac{1}{3}(u_{1}F_{1}+H+H'),
\]
and equality $\G^{2}=-2$ implies that 
\begin{equation}
u_{1}^{2}-ka'^{2}=1.\label{eq:NewPF}
\end{equation}
Since $a'\leq2y_{0}$ and $(x_{0},y_{0})$ is a fundamental solution
of \ref{eq:Pell-Fermat-4k}, one has $a'=2y_{0},\,u_{1}=x_{0}$ and
$\G$ has the form 
\[
\Gamma=\frac{2}{3}y_{0}L-\frac{1}{3}\left(x_{0}(A_{1}+2B_{1})+H+H'\right),
\]
with $H,H'$ in $\{F_{j},G_{j}\,|\,j\geq2\}$ such that $HH'=0$.
Since cases $v_{1}=0$ and $v_{1}=1$ are disjoint, again, by exchanging
the role of $A_{1}$ and $B_{1}$ and by considering 
\[
\Gamma'=\frac{2}{3}y_{0}L-\frac{1}{3}\left(x_{0}(2A_{1}+B_{1})+H_{1}+H_{1}'\right),
\]
and the difference $\G-\G'$, we see that $\G$ or $\G'$ is not in
the Néron-Severi lattice. 

\vspace*{1mm}

\textbf{Case 2) All the coefficients are integers.} From the previous
discussion, we can suppose that $a$ is not in $\frac{1}{3}\ZZ\setminus\ZZ$
(here $L^{2}=0\mod6$). Then by Corollary \ref{cor:Structure-of-NS}
all the coefficients of $\G$ are integers, which implies that $S$
is divisible by $3$, and since $S<3$, we obtain that $S=0$, moreover,
using Equation \ref{eq:Important-K-L2eq6k}, one has $v_{1}\in\{0,1\}$.
We thus obtain that 
\[
\G=aL-\frac{1}{3}(u_{1}F_{1}+v_{1}G_{1})
\]
with 
\begin{equation}
6ka^{2}-\frac{2}{3}(u_{1}^{2}+u_{1}v_{1}+v_{1}^{2})=-2,\label{eq:Truc1}
\end{equation}
where we recall that $L^{2}=6k$. 

\textbf{Sub-case 2) a) }Suppose $v_{1}=0$, then Equation \eqref{eq:Truc1}
is equivalent to $u_{1}^{2}=3(1+3ka^{2})$, which has no solutions. 

\textbf{Sub-case 2) b) }Suppose that $v_{1}=1$. By defining $U=\frac{1}{3}(2u_{1}+1)$,
Equation \eqref{eq:Truc1} is equivalent to 
\[
U^{2}-4ka^{2}=1.
\]
Since by \eqref{eq:Important-K-L2eq6k}, $a\leq y_{0}$ and $(x_{0},y_{0})$
is the fundamental solution of the above Pell-Fermat equation, we
obtain that $a=y_{0}$, $u_{1}=\frac{1}{2}(3x_{0}-1)$ and 
\[
\G=y_{0}L-(\tfrac{1}{2}(x_{0}+1)A_{1}+x_{0}B_{1})=B_{1}'.
\]
That finishes the proof of Proposition \ref{prop:L-is-Big-and-Nef-4k}. 

\vspace*{1mm}

As for the case $L^{2}=2\mod6$, we get : 
\begin{cor}
\label{cor:B1'-Is-(-2)-case-4k}Suppose that $(x_{0},y_{0})$ is the
fundamental solution of the Pell-Fermat equation \ref{eq:Pell-Fermat-4k}
and that $L^{2}=6\text{ or }12\mod18$ or $L^{2}=0\mod18$ and $3|y_{0}$.
The divisor $B_{1}'=y_{0}L-(\tfrac{1}{2}(x_{0}+1)A_{1}+x_{0}B_{1})$
is the class of a $\cu$-curve. \\
 The curves $B_{1}$ and $B_{1}'$ are the unique $\cu$-curves in
the lattice generated by $L,A_{1},B_{1}$ which have intersection
one with $A_{1}$. \\
 Up to exchanging the role of the curves $A_{1},B_{1}$, the same
result holds true for $L^{2}=0\mod18$ and $3\not|y_{0}$. 
\end{cor}

\begin{proof}
For proving that $B_{1}'$ is a $\cu$-curve we use the same argument
as in the proof of Corollary \ref{cor:B1'-Is-(-2)-curve}. Let $\tilde{B}_{1}$
be a $\cu$-curve such that $\tilde{B}_{1}A_{1}=1$. Suppose $\tilde{B}_{1}\neq B_{1}$,
there exists $(x,y)$ with $x>0,y>0$ solution of the Pell-Fermat
equation \ref{eq:Pell-Fermat-4k} such that 
\[
\tilde{B}_{1}=yL-\left(\frac{1}{2}(x+1)A_{1}+xB_{1}\right).
\]
One has $\tilde{B}_{1}B_{1}'<0$ if and only if 
\[
4kyy_{0}<xx_{0}+\frac{1}{3}.
\]
We proceed as in the proof of Corollary \ref{cor:B1'-Is-(-2)-curve}
and obtain that this inequality holds for any such $(x,y$), therefore
$(x,y)=(x_{0},y_{0})$. 
\end{proof}
%\newpage

\section{Existence of two generalized Kummer structures}

\subsection{A theoretical approach\label{subsec:A-theoretical-approach}}

Let $X$ be a generalized Kummer surface, we keep the notations as
before, in particular the polarization $L$ generates the orthogonal
complement to the $18$ curves $A_{1},\dots,B_{9}$. Through this
section we suppose that $6L^{2}$ is not a square (when $L^{2}=2\mod6$,
that assumption is always satisfied). For $L^{2}=2\mod6$ (respectively
for $L^{2}=0\mod6$), let $(x_{0},y_{0})$ be the fundamental solution
of the Pell-Fermat equation 
\[
x^{2}-12ty^{2}=1,
\]
for $t$ such that $L^{2}=2t$ (respectively 
\[
x^{2}-4ky^{2}=1,
\]
for $k$ such that $L^{2}=6k$). We remark that $x_{0}^{2}=1\mod12t$
(respectively $x_{0}^{2}=1\mod4k$). Let $B_{1}'$ be the $\cu$-class
\[
\begin{array}{l}
B_{1}'=3y_{0}L-(\frac{1}{2}(x_{0}+1)A_{1}+x_{0}B_{1}),\end{array}
\]
(respectively 
\[
B_{1}'=y_{0}L-(\frac{1}{2}(x_{0}+1)A_{1}+x_{0}B_{1}))\,\text{).}
\]
Let us suppose that $B_{1}'$ is the class of a $\cu$-curve. This
is guarantied by Corollaries \ref{cor:B1'-Is-(-2)-curve} and \ref{cor:B1'-Is-(-2)-case-4k}
if $L^{2}\neq0\mod18$, or if $L^{2}=0\mod18$ and $3|y_{0}$, or
up to exchanging the role of $A_{1}$ and $B_{1}$ if $L^{2}=0\mod18$
and $3\not|y_{0}$. Then, we know two generalized Nikulin configurations
\[
\mathcal{C}=\{A_{1},B_{1},\dots,A_{9},B_{9}\},\,\mathcal{C}'=\{A_{1},B_{1}',A_{2},B_{2},\dots,A_{9},B_{9}\}.
\]
Let us prove the following: 
\begin{thm}
\label{thm:MainProved}Suppose that $x_{0}\neq\pm1\mod2t$ (respectively
$x_{0}\neq\pm1\mod2k$). There is no automorphism sending the configuration
$\mathcal{C}$ to the configuration $\mathcal{C}'$. As a consequence,
there are (at least) two generalized Kummer structures on the generalized
Kummer surface $X$. 
\end{thm}

\begin{proof}
Let us suppose that such an automorphism $g$ sending $\mathcal{C}$
to $\mathcal{C}'$ exists. The automorphism $g$ induces an isometry
on $\NS X)$, it therefore sends the orthogonal complement of $\mathcal{C}$
to the orthogonal complement of $\mathcal{C}'$. Since $L,\,L'$ are
the positive generators of these complements, it maps $L$ to $L'$.
Suppose that $L$ is such that $L^{2}=2\mod6$. We recall that $L'=x_{0}L-2ty_{0}(A_{1}+2B_{1})$
and that by Section \ref{subsec:Construction-of-theNS}, the Néron-Severi
lattice is 
\[
\NS X)=\ZZ L\oplus\mathcal{K}{}_{3},
\]
therefore $\frac{1}{2t}L$ is in the dual of $\NS X)$ (we recall
that $L^{2}=2t$). Since we know that $g(L)=L'$, the action of $g$
on the class of $\frac{1}{2t}L$ in the discriminant group is 
\[
g^{*}(\frac{1}{2t}L)=\frac{x_{0}}{2t}L-y_{0}(A_{1}+2B_{1})=\frac{x_{0}}{2t}L\in\NS X)^{\vee}/\NS X).
\]
However, the action of an automorphism on the discriminant group must
be $\pm$ identity (see e.g. \cite[Section 8.1]{Shimada}). Since
the hypothesis is that $x_{0}\neq\pm1\mod2t$, such a $g$ does not
exist.\\
 Suppose that $L^{2}=0\mod6$, i.e.\,$L^{2}=6k$ ($k\in\NN^{*}$),
then $L'=x_{0}L-2ky_{0}(A_{1}+2B_{1})$. The Néron-Severi lattice
is generated by the lattice $\ZZ L\oplus\calK_{3}$ (see section \ref{subsec:Construction-of-theNS})
and by a vector $\tfrac{1}{3}(L+v_{6k})$, $v_{6k}\in\calK_{3}$.
The class $\frac{1}{2k}L$ is thus in the dual lattice of $\NS X)$.
The action of $g$ on the class of $\frac{1}{2k}L$ in the discriminant
group is 
\[
g^{*}(\frac{1}{2k}L)=\frac{x_{0}}{2k}L-y_{0}(A_{1}+2B_{1})=\frac{x_{0}}{2k}L.
\]
Again, since we supposed that $x_{0}\neq\pm1\mod2k$, this is impossible. 
\end{proof}
\begin{rem}
Suppose for simplicity that $L^{2}=2\text{ mod \ensuremath{6}}$.
One could play again the same game: pick-up a ${\bf A}_{2}$ configuration
$C_{1},D_{1}$ in $\mathcal{C}'$, then Corollary \ref{cor:B1'-Is-(-2)-curve}
implies that $D_{1}'=3y_{0}L-(\frac{1}{2}(x_{0}+1)C_{1}+x_{0}D_{1})$
is irreducible and we obtain in that way a new $9{\bf A_{2}}$-configuration
$\mathcal{C}''$, with orthogonal complement $L''$. Again there is
no automorphism sending $\mathcal{C}'$ to $\mathcal{C}''$. However,
the coefficient on $L$ of $L''$ will be $x_{0}^{2}$, which is congruent
to $1$ modulo $2t$, therefore there could be an automorphism sending
$\mathcal{C}$ to $\mathcal{C}''$, and in fact, in all the tested
cases, computations show that this always happens. 
\end{rem}

\begin{example}
The first $L^{2}$ for which Theorem \ref{thm:MainProved} provides
two generalized Kummer structures on the generalized Kummer surface
are: 
\begin{equation}
20,44,68,84,92,104,110,116,120,126,132,140,164,168,176,188.\label{eq:ExamplesLess200}
\end{equation}
\end{example}

The following infinite series of examples was given to us by Olivier
Ramaré: 
\begin{example}
Let $k$ be an integer, and let $a=8+12k$. Then $t=6+17k+12k^{2}$
is such that $a^{2}+a=12t$. The pair $(2a+1,2)$ is the fundamental
solution of the Pell-Fermat equation $x^{2}-12ty^{2}=1.$ One can
check that moreover $2a+1\neq\pm1\mod2t$ and therefore we can apply
Theorem \ref{thm:MainProved} for such $t$'s. 
\end{example}

The next Section suggests that the criteria in Theorem \ref{thm:MainProved}
for having two generalized Kummer structures is quite sharp.

\subsection{A computational approach\label{subsec:A-computational-approach}}

Suppose that the polarization $L$ on the generalized Kummer surface
$X$ is such that $6L^{2}$ is not a square and $L^{2}\neq0\mod18$
so that the curve $B_{1}'$ obtained in Corollaries \ref{cor:B1'-Is-(-2)-curve}
and \ref{cor:B1'-Is-(-2)-case-4k} is irreducible. Then, we know two
generalized Nikulin configurations 
\[
\mathcal{C}=\{A_{1},B_{1},\dots,A_{9},B_{9}\},\,\mathcal{C}'=\{A_{1},B_{1}',A_{2},B_{2},\dots,A_{9},B_{9}\}.
\]
Let $L$ and $L'$ be the big and nef generators of the orthogonal
complements of $\mathcal{C}$ and $\mathcal{C}'$ respectively. Suppose
that there is an automorphism $g$ sending $\mathcal{C}$ to $\mathcal{C}'$.
Then it induces an isometry $g^{*}$ of the lattice $\NS X)$, in
particular it sends $L$ to $L'$. Using the Torelli Theorem for K3
surfaces, one can test all linear maps 
\[
\psi:\NS X)\otimes\QQ\to\NS X)\otimes\QQ
\]
which satisfies to the following conditions:\\
 i) it sends the $\QQ$-base $\{L\}\cup\mathcal{C}$ to the $\QQ$-base
$\{L'\}\cup\mathcal{C}'$ and sends ${\bf A}_{2}$-configurations
in $\mathcal{C}$ to ${\bf A}_{2}$-configurations in $\mathcal{C}'$,
\\
 ii) it preserves the Néron-Severi lattice\\
 iii) it acts on the discriminant group of $\NS X)$ by $\pm Id$,\\
 iv) it sends an ample class to an ample class. \\
 About that last point, we remark that it is always satisfied by such
a $\psi$. Indeed, the nef divisor $L'$ generates the orthogonal
complement of a $9{\bf A}_{2}$ configuration $\cC'$, thus by Section
\ref{subsec:Polarizations}, for $u_{0}\geq4$, the divisor 
\[
D'=u_{0}L'-\sum_{C\in\mathcal{C}'}C
\]
is ample, but $D'$ is also the image by $\psi$ of the ample class
$D=u_{0}L-\sum_{C\in\mathcal{C}}C$.

Let us change the notations and define $C_{1}=A_{1},D_{1}=B_{1}'$,
$C_{j}=A_{j},D_{j}=B_{j}$ for $j\geq2$. There are 
\[
9!2^{9}=185794560
\]
maps $\psi$ satisfying condition i): $9!$ is for the number of maps
$\{A_{k},B_{k}\}\to\{C_{\s(k)},D_{\s(k)}\}$, where $\s\in S_{9}$
is a permutation, and this is times $2^{9}$, because one must choose
to send $A_{k}$ either to $C_{\s(k)}$ or to $D_{\s(k)}$. It can
be quite long to sort among them the maps that satisfy ii) and iii),
because one must deal with rank $19$ matrices.

However, recall that from Corollary \ref{cor:Structure-of-NS}, the
divisors supported only on the ${\bf A}_{2}$-blocs $A_{j},B_{j}$
($j\in\{1,\dots,9\}$) have restrictions: they form in $\NS X)\otimes\ZZ/3\ZZ$
a group isomorphic to $(\ZZ/3\ZZ)^{3}$, each with support on either
$6$ blocs of ${\bf A}_{2}$ ($12$ such words) or on $9$ blocs (2
words) or $0$ (thus a total of $2\times12+2\times1+1=3^{3}$ classes
which are $3$-divisible). With a computer, it is not difficult to
obtain the set $\text{BL}_{12}$ (respectively $\text{BL}_{12}'$)
of $6$ blocs which are $3$-divisible for $\cC$ (respectively for
$\cC'$). Since $\psi$ is an isometry it must send $3$-divisible
sets with support on $\mathcal{C}$ to $3$-divisible sets with support
on $\mathcal{C}'$. 

Using a computer, one can find in a few seconds the set of permutations
$\s$ such that if $\sum a_{i}(A_{i}+2B_{i})\in\text{BL}_{12}$ is
$3$-divisible, then $\sum a_{i}(C_{\s(i)}+2D_{\s(i)})$ is $3$-divisible.
That leaves a set of $432$ permutation, instead of $9!$. Therefore
the number of possibilities is divided by $800$, which makes the
computation for checking conditions ii) to iv) above last a few minutes
only. 

By these computations, we obtain the following result which gives
a more precise result (with an independent proof) than Theorem \ref{thm:MainProved},
but only for polarizations $L$ such that $L^{2}$ is below some bound: 
\begin{thm}
\label{smallL} Let $X$ be a generalized K3 surface polarized with
$L$, such that $L^{2}<200$ and $L^{2}\neq0\mod18$.\\
 There is an automorphism $\s$ sending the generalized Nikulin configuration
$\mathcal{C}\text{ to }\mathcal{C}'$ if and only if $L^{2}$ is not
in the list \eqref{eq:ExamplesLess200}. 
\end{thm}

Note that one can often choose $\sigma$ of order $2$, but for cases
$L^{2}=42,48$ all the automorphisms sending $\mathcal{C}$ to $\mathcal{C}'$
have infinite order. For the cases $L^{2}=36\text{ or }180$, provided
that the curve $B_{1}'$ constructed in Section \ref{subsc:B1-Irred-t=00003D00003D00003D0mod3}
is irreducible, there are also two generalized Kummer structures.
%%%%%%%%%%%%%%%%%%%%%%%%%%%%%%%%%%%%%%%%%%%%%%%%%%%%%%%%%%%%%%%%%%%%%%%%%%%%%%%%%%%%%%%%%%%%%%%%%

\section{Examples}

\subsection{A birational models of $X$ with $9{\bf A}_{2}$ singularities}

We keep the notations: $X$ is a generalized Kummer surface with Picard
number $19$, $A_{1},\dots,B_{9}$ is a $9{\bf A}_{2}$-configuration,
and $L$ is the nef divisor that generates the orthogonal complement
of these $18$ curves (see proof of Proposition \ref{ample}). 
\begin{prop}
Suppose $L^{2}>2$. The linear system $|L|$ induces a morphism $\varphi_{L}:X\to\PP^{\frac{L^{2}}{2}+1}$
which is an embedding outside the $\cu$-curves $A_{1},B_{1},\dots,A_{9},B_{9}$
and maps these curves to $9$ cusps. 
\end{prop}

\begin{proof}
We know that $L$ is nef and big. If $|L|$ has base points, then
\[
L=uF+\G,
\]
where $F$ is an elliptic curve and $\G$ is an irreducible $(-2)$-curve
such that $F\G=1$ (see \cite[Section 3.8]{Reid}). Moreover since
$L$ is big and nef we have 
\[
0\leq L\G=u-2
\]
so that $u>0$. If $\G\neq A_{k},B_{k}$ we have $FA_{k}\geq0$, $\G A_{k}\geq0$
and we compute 
\[
0=LA_{k}=uFA_{k}+\G A_{k}
\]
Since $u>0$ we obtain $FA_{k}=\G A_{k}=0$ similarly $FB_{k}=\G B_{k}=0$
so that $\G$ is in the orthogonal complement of $A_{1},B_{1},\ldots,A_{9},B_{9}$,
which is clearly impossible. If $\G=A_{k}$ we have $L=uF+A_{k}$
hence 
\[
0=LA_{k}=u-2
\]
which gives $u=2$, now $L=2F+A_{k}$ which has $L^{2}=2$, which
contradicts the assumption, similarly for $B_{k}$. In conclusion
$|L|$ is base-point-free. %%%%%%%%%%%%%%%%%%%%%%%%%%%%%%%%%%
Suppose that $|L|$ is hyperelliptic (see \cite{SD}). Since $L$
is primitive, one cannot have $L=2D_{2}$ with $D_{2}^{2}=2$. Suppose
there is an elliptic curve $F$ such that $FL=2$. Write 
\[
F=aL-\sum\a_{i}A_{i}+\b_{i}B_{i},
\]
with $a\geq0$, $a,\alpha_{i},\beta_{i}\in\tfrac{1}{3}\ZZ$, then
\[
2=FL=aL^{2}.
\]
Since $L^{2}>2$, this is possible only if $L^{2}=6$. But in that
case J. Bertin and P. Vanhaecke \cite{BertinVan} proved that $\varphi_{L}$
is an embedding outside the $\cu$-curves $A_{1},B_{1},\dots,A_{9},B_{9}$
(see the description below). 
\end{proof}

\subsection{Case $L^{2}=2$}

In the case $L^{2}=2$, the generalized Kummer surface $X=\Km(A,G)$
is the double cover of the plane ramified on a sextic curve with $9A_{2}$
singularities. This double cover were first studied by Ch. Birkenhake
and H. Lange in \cite{LanBir}. In \cite{KRS}, the two authors of
the present paper and D. Kohel determined several $9\bA_{2}$-configurations
on $X$ related to a special configuration of conics in the plane.
%the {\it chilean configuration of conics}(see \cite{DLPU}). 
If $L^{2}=2$ by Section \ref{subsec:Exisentence-and-unicity-of-B1}
we have $B_{1}'=6L-(4A_{1}+7B_{1})$. Since $LB_{1}'=12$, the curve
$B_{1}'$ is sent to a singular curve of degree $6$ in the plane,
which passes through the cusp obtained by the contraction of $A_{1},B_{1}$.
As shown in \cite{Roulleau2} (see also Corollary \ref{smallL}),
up to automorphism of the K3 surface, there is only one $9\bA_{2}$-configuration,
so that there exists an automorphism sending the configuration $\mathcal{C}$
to the configuration $\mathcal{C}'$; observe that clearly this automorphism
is not the covering involution.

\subsection{Case $L^{2}=6$}

The polarization $L^{2}=6$ exhibits the K3 surface as a complete
intersection of a quadric and a cubic in $\PP^{4}$ with $9{\bf A}_{2}$
singularities, see the paper by J. Bertin and P. Vanhaecke \cite{BertinVan}
for more details and the equations. Observe that in this case the
Pell--Fermat equation $x^{2}-4y^{2}=1$ has no solution so that we
can not apply our construction. By \cite{CR} in this case there is
only one generalized Nikulin configuration.

\subsection{Case $L^{2}=20$}

This is the first example for which our construction gives two non--equivalent
Kummer structures (see Example \ref{eq:ExamplesLess200}), so we study
it more in detail. As before let $L$ be the class such that $L^{2}=20$,
let $A_{1},B_{1},\dots,A_{9},B_{9}$ be the $9\bA_{2}$-configuration
orthogonal to $L$. In this case the class $B_{1}'$ is 
\[
B_{1}'=3L-(6A_{1}+11B_{1})
\]
and we have shown that $\mathcal{C}=\{A_{1},B_{1},\dots,A_{9},B_{9}\}$
and $\mathcal{C}'=\{A_{1},B_{1}',\dots,A_{9},B_{9}\}$ are not equivalent.
The projective model determined by $L$ is a surface in $\mathbb{P}^{11}$,
we describe here another projective model as a double plane ramified
on a special sextic curve. By Proposition \ref{ample}, the class
\[
D_{2}=L-\sum_{k=1}^{9}\left(A_{k}+B_{k}\right)
\]
is ample with $D_{2}^{2}=2$ and $D_{2}A_{k}=D_{2}B_{k}=1$, $k\in\{1,\ldots,9\}$.
The $\cu$-classes 
\[
E_{k}=D_{2}-A_{k},\,\,F_{k}=D_{2}-B_{k},\,\,k\in\{1,\dots,9\}
\]
are also of degree $1$ for $D_{2}$, hence are classes of $\cu$-curves.
Moreover one can check easily that: 
\begin{prop}
The $18$ $\cu$-curves $E_{1},F_{1},\dots,E_{9},F_{9}$ form a $9\bA_{2}$-configuration. 
\end{prop}

Suppose that $D_{2}$ has base points. Then there exist an elliptic
curve $F$ and an irreducible $\cu$-curve $\G$, \cite[Section 3.8]{Reid}
such that $D_{2}=2F+\G$ and $F\G=1$. But then $\G D_{2}=0$, which
is a contradiction since $D_{2}$ is ample. Therefore, with the previous
notations: 
\begin{prop}
The generalized Kummer surface is a double cover of $\PP^{2}$ branched
over a smooth sextic curve $C_{6}$ which has $18$ tritangent lines,
which are the images of the $18$ couples of curves $(E_{k},A_{k})$
and $(F_{k},B_{k})$ for $k\in\{1,\dots,9\}$. 
\end{prop}

Plane sextic curves with several tritangents were studied by A. Degtyarev
in \cite{Deg}. Using the Néron-Severi lattice and Vinberg's algorithm
\cite{Vinberg}, one can compute moreover, that there are $1728$
$6$-tangent conics to $C_{6}$ and $67212$ rational cuspidal curves
which are tangent to $C_{6}$, with a cusp on $C_{6}$. Also we have: 
\begin{thm}
The automorphism group $G_{36}$ preserving the polarization $D_{2}$
is isomorphic to 
\[
\ZZ_{2}\times(\ZZ_{3}\rtimes S_{3}),
\]
it has order $36$. It is generated by the involution $\sigma$ of
the double cover and a symplectic group of automorphisms $G_{18}$
of order $18$; $\s$ generates the center of $G_{36}.$ 
\end{thm}

\begin{proof}
This is obtained by a computation as explained in Section \ref{subsec:A-computational-approach},
using the divisibility relations among the $(-2)$-curves. The algorithm
in Section \ref{subsec:A-computational-approach} is also able to
construct the automorphisms if such exists, one just have to check
that the hypothesis of the Torelli Theorem are satisfied. 
\end{proof}
The involution $\s$ is such that $\s(A_{k})=E_{k},\,\s(B_{k})=F_{k}$,
so that the two $9\bA_{2}$-configurations $E_{1},F_{1},\dots,E_{9},F_{9}$
and $\mathcal{C}=\{A_{1},B_{1},\dots,A_{9},B_{9}\}$ are equivalent.

The orbit of $A_{1}$ under $G_{36}$ is $\{A_{k},E_{k}\,|\,1\leq k\leq9\}$
and the orbit of $B_{1}$ is $\{B_{k},F_{k}\,|\,1\leq k\leq9\}$.
Let $L_{k}$ (respectively $L_{k}'$) be the image of $A_{k}$ (respectively
$B_{k}$) by the double cover map $X\to\PP^{2}$. The group $\ZZ_{3}\rtimes S_{3}$
acts on the plane and the orbit of $L_{1}$ (resp. $L_{1}'$) is $\{L_{k}\,|\,1\leq k\leq9\}$
(resp. $\{L_{k}'\,|\,1\leq k\leq9\}$).

The general abelian surfaces $A$ such that $X=\Km_{3}(A)$ are simple
\cite{Roulleau2}. 

\vspace{4mm}
Xavier Roulleau, \\
Université d'Angers, \\
CNRS, LAREMA, SFR MATHSTIC, \\
F-49000 Angers, France 

Xavier.Roulleau@univ-angers.fr

\noindent \vspace{0.2cm}
 %\\ 

Alessandra Sarti\\
 Université de Poitiers\\
 Laboratoire de Mathématiques et Applications,\\
 UMR 7348 du CNRS, \\
 TSA 61125 \\
 11 bd Marie et Pierre Curie, \\
 86073 Poitiers Cedex 9, France

Alessandra.Sarti@math.univ-poitiers.fr 
\begin{verbatim}
http://www-math.sp2mi.univ-poitiers.fr/~sarti/

\end{verbatim}

\end{document}